\documentclass[reqno, 11pt]{amsart}

\usepackage{varioref}
\usepackage[
    eientei
    , +maths
    , +tikz
    , scriptStyle = euler
    , margins = reduced
    , numberTheoremsWithin = subsection
]{kappak} 

\usepackage{enumitem}
\usepackage{mathbbol} 
\usepackage{MnSymbol} 
\usepackage{tikz-cd}
\usepackage{todonotes}

\makeindex

\DeclareSymbolFontAlphabet{\mathbb}{AMSb}
\DeclareSymbolFontAlphabet{\mathbbold}{bbold}

\usetikzlibrary{calc, cd, decorations.pathmorphing, positioning}
\tikzset{treeedge/.style = {above, sloped}}
\tikzset{treenode/.style = {circle, minimum size = 3pt, inner sep = 0, draw, fill}}

\input{macros.tex}

\renewcommand{\ulcorner}{\lefthalfcap}
\renewcommand{\lrcorner}{\righthalfcup}

\newcategory{\Comb}{}{C}{omb}
\newcategory{\Cart}{}{C}{art}

\renewcommand{\bbDelta}{\mathbbold{\Delta}}
\renewcommand{\bbLambda}{\mathbbold{\Lambda}}
\renewcommand{\bbOmega}{\mathbbold{\Omega}}

\newcommand{\presheaf}{\mathrm{\catPP sh}}
\newcommand{\Psh}[1]{\presheaf (#1)}
\renewcommand{\PshO}{{\Psh\bbOO}}

\newcommand{\PshLambda}{\Psh\bbLambda}

\newcommand{\Oalg}[2]{\Alg_{#1, #2}}

\newcategory{\elTree}{el}{T}{ree}
\DeclareMathOperator{\elt}{elt}
\usepackage{pict2e}
\makeatletter
\newcommand{\adjunction}{
    \mathrel{\vcenter{
        \offinterlineskip\m@th
        \ialign{
            \hfil$##$\hfil\cr
            \longrightarrow\cr
            \noalign{\kern-.2ex}
            \begingroup\setlength\unitlength{.15em}
            \begin{picture} (1, 1)
                \roundcap
                \polyline (0, 0) (1, 0)
                \polyline (0.5, 0) (0.5, 1)
            \end{picture}
            \endgroup\cr
            \longleftarrow\cr
        }
    }}
}
\newcommand{\longadjunction}{
    \mathrel{\vcenter{
        \offinterlineskip\m@th
        \ialign{
            \hfil$##$\hfil\cr
            \relbar\joinrel\relbar\joinrel\relbar\joinrel\relbar\joinrel\relbar\joinrel\rightarrow\cr
            \noalign{\kern-.2ex}
            \begingroup\setlength\unitlength{.15em}
            \begin{picture} (1, 1)
                \roundcap
                \polyline (0, 0) (1, 0)
                \polyline (0.5, 0) (0.5, 1)
            \end{picture}
            \endgroup\cr
            \leftarrow\joinrel\relbar\joinrel\relbar\joinrel\relbar\joinrel\relbar\joinrel\relbar\cr
        }
    }}
}
\makeatother
\newcommand{\polynomialinline}[4]{#1 \longleftarrow #2 \longrightarrow #3 \longrightarrow #4}

\newcommand{\longleadsto}{
    \begin{tikzcd} [ampersand replacement = \&]
        ~
            \ar[r, rightsquigarrow] \&
        ~
    \end{tikzcd}
}

\newcommand{\longhookrightarrow}{
    \mathrel{
        \hookrightarrow
        \!\!\!\!
        \tikz \fill[white] (0, 0) rectangle (.1,.4);
        \!\!\!\!\!\!
        \relbar
        \!\!
        \rightarrow
    }
}
\newcommand{\inv}{^{-1}}

\newcommand{\qqquad}{\qquad\qquad}

\newcommand{\noop}[1]{}

\title{Opetopic algebras I: Algebraic structures on opetopic sets}
\date{October 2019}

\author[C. Ho Thanh]{C\'edric Ho Thanh}
\address{
    Institut de Recherche en Informatique Fondamentale (IRIF),
    Universit\'e Paris Diderot,
    France
}
\email{cedric.hothanh@irif.fr}
\urladdr{hothanh.fr/cedric}

\author[C. Leena Subramaniam]{Chaitanya Leena Subramaniam}
\address{
    Institut de Recherche en Informatique Fondamentale (IRIF),
    Universit\'e Paris Diderot,
    France
}
\email{chaitanya@irif.fr}
\urladdr{www.irif.fr/~chaitanya}


\begin{document}

\begin{abstract}
    We define a family of structures called ``opetopic algebras'', which are
    algebraic structures with an underlying opetopic set. Examples of such are
    categories, planar operads, and Loday's combinads over planar trees.
    Opetopic algebras can be defined in two ways, either as the algebras of a
    ``free pasting diagram'' parametric right adjoint monad, or as models of a
    small projective sketch over the category of opetopes. We define an
    opetopic nerve functor that fully embeds each category of opetopic algebras
    into the category of opetopic sets. In particular, we obtain fully faithful
    opetopic nerve functors for categories and for planar coloured
    $\Set$-operads.

    This paper is the first in a series aimed at using opetopic spaces as models
    for higher algebraic structures.
\end{abstract}

\maketitle

\begin{small}
    \tableofcontents
\end{small}

\section{Introduction}

This paper deals with algebraic structures whose operations have \emph{higher
dimensional ``tree-like'' arities}. As an example in lieu of a definition, a
category $\catCC$ is an algebraic structure whose operation of
\emph{composition} has as its inputs, or ``arities'', sequences of composable
morphisms of the category. These sequences can can be seen as \emph{filiform}
or \emph{linear} trees. Moreover, each morphism of $\catCC$ can itself be seen
as an operation whose arity is a single point (i.e. an object of $\catCC$). A
second example, one dimension above, is that of a planar coloured $\Set$-operad
$\calPP$ (a.k.a. a nonsymmetric multicategory), whose operation of composition
has planar trees of operations (a.k.a. multimorphisms) of $\calPP$ as arities.
Moreover, the arity of an operation of $\calPP$ is an ordered list (a filiform
tree) of colours (a.k.a. objects) of $\calPP$. Heuristically extending this
pattern leads one to presume that such an algebraic structure one dimension
above planar operads should have an operation of composition whose arities are
trees of things that can themselves be seen as operations whose arities are
planar trees. Indeed, such algebraic structures are precisely the \emph{$\bbPP
\bbTT$-combinads in $\Set$} (combinads over the combinatorial pattern of planar
trees) of Loday \cite{Loday2012a}.

\vspace{.5em}
\begin{center}
    \begin{tabular}{c|ccccc}
      Structure & $\substack{\text{Sets} \\ = 0\text{-algebras}}$ &
      $\substack{\text{Categories} \\ = 1\text{-algebras}}$ &
      $\substack{\text{Operads} \\ = 2\text{-algebras}}$ &
      $\substack{\bbPP\bbTT-\text{combinads} \\ = 3\text{-algebras}}$ &
      $\cdots$ \\
    \hline
    Arity & $\substack{\{ * \} \\ = 1\text{-opetopes}}$ &
        $\substack{\text{Lists} \\ = 2\text{-opetopes}}$ &
        $\substack{\text{Trees} \\ = 3\text{-opetopes}}$ & $4$-opetopes &
        $\cdots$
    \end{tabular} \end{center}
\vspace{.5em}

The goal of this article is to give a precise definition of the previous
sequence of algebraic structures.

\subsection{Context}

It is well-known that higher dimensional tree-like arities are encoded by the
opetopes (\emph{ope}ration poly\emph{topes}) of Baez and Dolan \cite{Baez1998},
which were originally introduced in order to give a definition of weak
$n$-categories and a precise formulation of the ``microcosm'' principle.

The fundamental definitions of \cite{Baez1998} are those of
\emph{$\calPP$-opetopic sets} and \emph{$n$-coherent $\calPP$-algebras} (for a
coloured symmetric $\Set$-operad $\calPP$), the latter being $\calPP$-opetopic
sets along with certain ``horn-filling'' operations that are ``universal'' in a
suitable sense. When $\calPP$ is the identity monad on $\Set$ (i.e. the
unicolour symmetric $\Set$-operad with a single unary operation),
$\calPP$-opetopic sets are simply called \emph{opetopic sets}, and $n$-coherent
$\calPP$-algebras are the authors' proposed definition of \emph{weak
  $n$-categories}.

While the coinductive definitions in \cite{Baez1998} of $\calPP$-opetopic sets and $n$-coherent
$\calPP$-algebras are straightforward and general, they have the
disadvantage of not defining a \emph{category} of $\calPP$-opetopes such that
presheaves over it are precisely $\calPP$-opetopic sets, even though this is
ostensibly the case. Directly defining the category of $\calPP$-opetopes turns
out to be a tedious and non-trivial task, and was worked out explicitly by
Cheng in \cite{Cheng2003,Cheng2004} for the particular case of the identity
monad on $\Set$, giving the category $\bbOO$ of opetopes.

The complexity in the definition of a category $\bbOO_\calPP$ of
$\calPP$-opetopes has its origin in the difficulty of working with a suitable
notion of \emph{symmetric tree}. Indeed, the objects of $\bbOO_\calPP$ are
trees of trees of ... of trees of operations of the symmetric operad $\calPP$,
and their automorphism groups are determined by the action of the
(``coloured'') symmetric groups on the sets of operations of $\calPP$.

However, when the action of the symmetric groups on the sets of operations of
$\calPP$ is free, it turns out that the objects of the category $\bbOO_\calPP$
are rigid, i.e. have no non-trivial automorphisms (this follows from
\cite[proposition 3.2]{Cheng2004}). The identity monad on $\Set$ is of course
such an operad, and this vastly simplifies the definition of $\bbOO$. Indeed,
such operads are precisely the (finitary) \emph{polynomial monads} in $\Set$,
and the machinery of polynomial endofunctors and polynomial monads developed in
\cite{Kock2011,Kock2010,Gambino2013} gives a very satisfactory definition of
$\bbOO$ \cite{HoThanh2018,Curien2018} which we review in
\cref{sec:opetopes,sec:polynomial-functors-and-monads}.

\subsection{Contributions}

The main contribution of the present article is to show how the polynomial
definition of $\bbOO$ allows, for all $k, n \in \bbNN$ with $k \leq n$, a
definition of \emph{$(k, n)$-opetopic algebras}, which constitute a full
subcategory of the category $\PshO$ of opetopic sets. More precisely, we show
that the polynomial monad whose set of operations is the set $\bbOO_{n+1}$ of
$(n+1)$-dimensional opetopes can be extended to a parametric right adjoint
monad whose algebras are the $(k, n)$-opetopic algebras. Important particular
cases are the categories of $(1, 1)$- and $(1, 2)$-opetopic algebras, which are
the categories $\Cat$ and $\Op_{\mathrm{col}}$ of small categories and coloured
planar $\Set$-operads respectively. Loday's combinads over the combinatorial
pattern of planar trees \cite{Loday2012a} are also recovered as $(1,
3)$-opetopic algebras.

We further show that each category of $(k, n)$-opetopic algebras admits a fully
faithful \emph{opetopic nerve functor} to $\PshO$. As a direct consequence of
this framework, we obtain commutative triangles of adjunctions
\begin{equation}
    \label{eq:triangle-adjunctions}
    \begin{tikzcd} [column sep = small, row sep = large]
        &
        \Cat
            \ar[dl, shift left = .4em, "N" below right, "\scriptstyle{\perp}"
            {sloped, above}, hook']
            \ar[dr, shift right = .4em, "N_u" below left,
            "\scriptstyle{\perp}" {sloped, above}, hook] &
        \\
        \PshO
            \ar[ur, shift left = .4em, "h" above left]
            \ar[rr, shift left = .4em, "h_!" above] &
        &
        \Psh\bbDelta
            \ar[ul, shift right = .4em, "u" above right]
            \ar[ll, shift left = .4em, "h^*" below, "\scriptstyle{\perp}" {sloped, above}]
        \end{tikzcd}
        \qquad
\begin{tikzcd} [column sep = small, row sep = large]
        &
        \Op_{\mathrm{col}}
            \ar[dl, shift left = .4em, "N" below right, "\scriptstyle{\perp}"
            {sloped, above}, hook']
            \ar[dr, shift right = .4em, "N_u" below left,
            "\scriptstyle{\perp}" {sloped, above}, hook] &
        \\
        \Psh{\bbOO_{\geq 1}}
            \ar[ur, shift left = .4em, "h" above left]
            \ar[rr, shift left = .4em, "h_!" above] &
        &
        \Psh\bbOmega ,
            \ar[ul, shift right = .4em, "u" above right]
            \ar[ll, shift left = .4em, "h^*" below, "\scriptstyle{\perp}" {sloped, above}]
        \end{tikzcd}
\end{equation}
where $\bbDelta$ is the category of simplices, $\bbOmega$ is the planar version
of Moerdijk and Weiss's category of \emph{dendrices} and $\bbOO_{\geq 1}$ is
the full subcategory of $\bbOO$ on opetopes of dimension $>0$. This gives a
direct comparison between the opetopic nerve of a category (resp. a planar
operad) and its corresponding well-known simplicial (resp. dendroidal) nerve.

This formalism seems to provide infinitely many types of $(k, n)$-opetopic
algebras. However this is not really the case, as the notion stabilises at the
level of combinads. Specifically, we show a phenomenon we call \emph{algebraic
trompe-l'\oe{}il}, where an $(k, n)$-opetopic algebra is entirely specified by
its underlying opetopic set and by a $(1, 3)$-opetopic algebra. In other words,
its algebraic data can be ``compressed'' into a $(1, 3)$-algebra (a combinad).
The intuition behind this is that  fundamentally, opetopes are just trees whose
nodes are themselves trees, and that once this is obtained at the level of
combinads, opetopic algebras can encode no further useful information.

\subsection{Outline}

We begin by recalling elements of the theory of polynomial functors and
polynomial monads in \cref{sec:polynomial-functors-and-monads}. This formalism
is the basis for the modern definition of opetopes and of the category of
opetopes \cite{Kock2010, Curien2018} that we survey in \cref{sec:opetopes}.
Section \ref{sec:algebraic-realization} contains the central constructions of
this article, namely those of opetopic algebras and coloured opetopic algebras,
as well as the definition of the opetopic nerve functor, which is a full
embedding of (coloured) opetopic algebras into opetopic sets. Section
\ref{sec:trome-loeil} is devoted to showing how the algebraic information
carried by opetopes turns out to be limited.

\subsection{Related work}

The $(k, n)$-opetopic algebras that we obtain are related to the $n$-coherent
$\calPP$-algebras of \cite{Baez1998} as follows: for $n \geq 1$, $(1,
n)$-opetopic algebras are precisely $1$-coherent $\calPP$-algebras for $\calPP$
the polynomial monad
$\polynomialinline{\bbOO_{n-1}}{E_n}{\bbOO_n}{\bbOO_{n-1}}$. We therefore do
not obtain all $n$-coherent $\calPP$-algebras with our framework, and this
means in particular that we cannot capture all the weak $n$-categories of
\cite{Baez1998} (except for $n=1$, which are just usual $1$-categories). This
is not unexpected, as weak $n$-categories are not defined just by
\emph{equations} on the opetopes of an opetopic set, but by its more subtle
\emph{universal} opetopes.

However, we are able to promote the triangles of
\eqref{eq:triangle-adjunctions} to Quillen equivalences of simplicial model
categories. This, along with a proof that opetopic spaces (i.e. simplicial
presheaves on $\bbOO$) model $(\infty,1)$-categories and planar
$\infty$-operads, will be the subject of the second paper of this series.

\subsection{Acknowledgments}

We are grateful to Pierre-Louis Curien for his patient guidance and reviews.
The second named author would like to thank Paul-Andr\'{e} Melli\`{e}s for
introducing him to nerve functors and monads with arities, Mathieu Anel for
discussions on Gabriel--Ulmer localisations, and Eric Finster for a discussion
related to \cref{prop:algebraic-realization:equivalence-P+-algebras}. The first
named author has received funding from the European Union's Horizon 2020
research and innovation program under the Marie Sklodowska--Curie grant
agreement No. 665850.

\subsection{Preliminary category theory}
\label{sec:preliminaries-category-theory}

We review relevant notions and basic results from the theory of locally
presentable categories. Original references are \cite{Gabriel1971,SGA4}, and
most of this material can be found in \cite{Adamek1994a}.




\subsubsection{Presheaves and nerve functors}
\label{sec:preliminaries-category-theory:prsh-nerve-functors}

For $\catCC \in \Cat$, we write $\Psh \catCC$\index{$\Psh\catCC$} for the
category of $\Set$-valued presheaves over $\catCC$, i.e. the category of
functors $\catCC^\op \longrightarrow \Set$ and natural transformations between
them. If $X \in \Psh\catCC$ is a presheaf, then its \emph{category of
elements}\index{category of elements} $\catCC / X$\index{$\catCC / X$|see
{category of elements}} is the comma category $\sfY / X$, where $\sfY : \catCC
\longhookrightarrow \Psh\catCC$ is the Yoneda embedding.

Let $\catCC \in \Cat$ and $F : \catCC \longrightarrow \catDD$ be a functor to a (not
necessarily small) category $\catDD$. Then the \emph{nerve functor associated
to $F$}\index{nerve} (also called the \emph{nerve of $F$}) is the functor $N_F
: \catDD \longrightarrow \Psh \catCC$\index{$N_F$|see {nerve}} mapping $d \in \catDD$ to
$\catDD(F -, d)$.
The functor $F$ is said to be \emph{dense}\index{dense functor} if for all $d
\in \catDD$, the colimit of $F / d \longrightarrow \catCC \xrightarrow{F}
\catDD$ exists in $\catDD$, and is canonically isomorphic to $d$. Equivalently,
$F$ is dense if and only if $N_F$ is fully faithful.

Let $i : \catCC \longrightarrow \catDD$ be a functor between small categories.
Then the precomposition functor $i^* : \Psh\catDD \longrightarrow \Psh\catCC$ has a left
adjoint $i_!$ and a right adjoint $i_*$, given by left and right Kan extension
along $i^\op$ respectively. If $i$ has a right adjoint $j$, then $i^* \dashv
j^*$, or equivalently, $i^* \cong j_!$. Note that the nerve of $i$ is the
functor $N_i = i^* \sfY_\catDD$, where $\sfY_\catDD : \catDD \longhookrightarrow \Psh\catDD$
is the Yoneda embedding. Recall that $i^*$ is the nerve of the functor
$\sfY_\catDD i$ and that $i_*$ is the nerve of the functor $N_i = i^*
\sfY_\catDD$, i.e. it is the nerve of the nerve of $i$.

\subsubsection{Orthogonality}
\label{sec:preliminaries-category-theory:orthogonality}

Let $\catCC$ be a category, and $l,r \in \catCC^\rightarrow$. We say that $l$
is \emph{left orthogonal to}\index{orthogonal} $r$ (equivalently, $r$ is
\emph{right orthogonal} to $l$), written $l \perp r$\index{$\perp$|see
{orthogonal}}, if for any solid commutative square as follows, there exists a
\emph{unique} dashed arrow making the two triangles commute (the relation
$\perp$ is also known as the \emph{unique lifting property}\index{unique
lifting property|see {orthogonal}}):
\[
    \begin{tikzcd}
        \cdot \ar[d, "l" left] \ar[r] & \cdot \ar[d, "r"] \\
        \cdot \ar[r] \ar[ur, dashed] & \cdot
    \end{tikzcd}
\]

If $\catCC$ has a terminal object $1$, then for all $X \in \catCC$, we write $l
\perp X$ if $l$ is left orthogonal to the unique map $X \longrightarrow 1$. Let
$\sfLL$ and $\sfRR$ be two classes of morphisms of $\catCC$. We write $\sfLL
\perp \sfRR$ if for all $l \in \sfLL$ and $r \in \sfRR$ we have $l \perp r$.
The class of all morphisms $f$ such that $\sfLL \perp f$ (resp. $f \perp
\sfRR$) is denoted $\sfLL^\perp$ (resp ${}^\perp \sfRR$).



\subsubsection{Localisations}

Let $\sfJJ$ be a class of morphisms of a category $\catCC$. Recall that the
\emph{localisation of $\catCC$ at $\sfJJ$}\index{localization} is a functor
$\gamma_\sfJJ : \catCC \longrightarrow \sfJJ\inv \catCC$ such that $\gamma_\sfJJ f$ is an
isomorphism for every $f \in \sfJJ$, and such that $\gamma_\sfJJ$ is universal
for this property. We say that $\sfJJ$ has the \emph{3-for-2}\index{3-for-2}
property when for every composable pair $\cdot \xrightarrow{f} \cdot \xrightarrow{g} \cdot$ of
morphisms in $\catCC$, if any two of $f$, $g$ and $gf$ are in $\sfJJ$, then so
is the third.


Assume now that $\catCC$ is a small category, and that $\sfJJ$ is a \emph{set}
(rather than a proper class) of morphisms of $\Psh\catCC$, and consider the
full subcategory $\catCC_\sfJJ \longhookrightarrow \Psh\catCC$ of all those $X \in
\Psh\catCC$ such that $\sfJJ \perp X$. A category is \emph{locally
presentable}\index{locally presentable category} if and only if it is
equivalent to one of the form $\catCC_\sfJJ$. The pair $({}^\perp
(\sfJJ^\perp), \sfJJ^\perp)$ forms an \emph{orthogonal factorisation
system}\index{orthogonal factorization system}, meaning that any morphism $f$
in $\Psh\catCC$ can be factored as $f = p i$, where $p \in \sfJJ^\perp$ and $i
\in {}^\perp (\sfJJ^\perp)$. Applied to the unique arrow $X \longrightarrow 1$, this
factorisation provides a left adjoint (i.e. a reflection) $a_\sfJJ : \Psh\catCC
\longrightarrow \catCC_\sfJJ$ to the inclusion $\catCC_\sfJJ \longhookrightarrow \Psh\catCC$.
Furthermore, $a_\sfJJ$ is the localisation of $\Psh\catCC$ at $\sfJJ$.
\footnote{The results of this paragraph still hold when $\Psh\catCC$ is
replaced by a locally $\kappa$-presentable category $\catEE$ and a set $\sfKK$
of $\kappa$-small morphisms of $\catEE$. We call a localisation of the form
$a_\sfKK : \catEE \longrightarrow \catEE_\sfKK$ the \emph{Gabriel--Ulmer
localisation}\index{Gabriel--Ulmer localisation} of $\catEE$ at $\sfKK$.}

With $\catCC$ and $\sfJJ$ as in the previous paragraph, the class of
\emph{$\sfJJ$-local isomorphisms}\index{local isomorphism} $\sfWW_\sfJJ$ is the
class of all morphisms $f \in \Psh\catCC^\rightarrow$ such that for all $X \in
\catCC_\sfJJ$, $f \perp X$ (that is, $\sfWW_\sfJJ = {}^\perp \catCC_\sfJJ$). It
is the smallest class of morphisms that contains $\sfJJ$, that satisfies the
3-for-2 property, and that is closed under colimits in $\Psh\catCC^\rightarrow$
\cite[theorem 8.5]{Gabriel1971}. Thus the localisation $a_\sfJJ$ is also the
localisation of $\Psh\catCC$ at $\sfWW_\sfJJ$. Furthermore, $\sfWW_\sfJJ$ is
closed under pushouts.




\subsubsection{Projective sketches}

A \emph{projective sketch}\index{projective sketch} is the data of a $\catCC
\in \Cat$ and a set $K$ of cones in $\catCC$. For $\catBB$ a category with all
limits, the category of \emph{models in $\catBB$ of $(\catCC, K)$} is the
category of functors $\catCC \longrightarrow \catBB$ that take each cone in $K$ to a limit
cone, and natural transformations between them. If $(\catCC, K)$ is a
projective sketch, then equivalently, $K$ can be seen as a set of cocones in
$\catCC^\op$, or as a set of sub-representables (subobjects of representables)
in $\Psh{\catCC^\op}$. Let $\catCC^\op _K \longhookrightarrow \Psh{\catCC^\op}$ be the full
subcategory of all $X \in \Psh{\catCC^\op}$ such that $K \perp X$. Then
$\catCC^\op_K$ is precisely the category of models in $\Set$ of the projective
sketch $(\catCC, K)$. Let $\kappa$ be a regular cardinal, and let $(\catCC, K)$
be a projective sketch in which each cone in $K$ is a $\kappa$-small diagram.
Then $\catCC^\op_K$ is a locally $\kappa$-presentable category, and the
inclusion $\catCC^\op_K \longhookrightarrow \Psh{\catCC^\op}$ preserves $\kappa$-filtered
colimits.


\section{Polynomial functors and polynomial monads}
\label{sec:polynomial-functors-and-monads}

We survey elements of the theory of polynomial functors, trees, and monads. For
more comprehensive references, see \cite{Kock2011, Gambino2013}.

\subsection{Polynomial functors}

\begin{definition}
    [Polynomial functor]
    \label{def:polynomial-functor}
    A \emph{polynomial (endo)functor $P$ over $I$}\index{polynomial
    endofunctor} is a diagram in $\Set$ of the form
    \begin{equation}
        \label{eq:polynomial-functor}
        \polynomialfunctor{I}{E}{B}{I.}{s}{p}{t}
    \end{equation}
    $P$ is said to be \emph{finitary}\index{finitary polynomial functor} if the
    fibres of $p:E \longrightarrow B$ are finite sets. We will always assume
    polynomial functors to be finitary.

    We use the following terminology for a polynomial functor $P$ as in
    \cref{eq:polynomial-functor}, which is motivated by the intuition that a
    polynomial functor encodes a multi-sorted signature of function symbols.
    The elements of $B$ are called the \emph{nodes}\index{nodes} or
    \emph{operations}\index{operations} of $P$, and for every node $b$, the
    elements of the fibre $E (b) \eqdef p^{-1} (b)$ are called the
    \emph{inputs}\index{inputs} of $b$. The elements of $I$ are called the
    \emph{colours}\index{colours} or \emph{sorts}\index{sorts} of $P$. For
    every input $e$ of a node $b$, we denote its colour by $s_e (b) \eqdef s
    (e)$.
    \[
        \tikzinput{polynomial-functors}{operation}
    \]
\end{definition}

\begin{definition}
    [Morphism of polynomial functor]
    \label{def:polynomial-endofunctor-morphism}
    A morphism from a polynomial functor $P$ over $I$ (as in
    \cref{eq:polynomial-functor}) to a polynomial functor $P'$ over $I'$ (on
    the second row) is a commutative diagram of the form
    \[
        \begin{tikzcd}
            I \ar[d, "f_0"'] &
            E \pullbackcorner \ar[r, "p"] \ar[d, "f_2" left] \ar[l, "s" above] &
            B \ar[r, "t"] \ar[d, "f_1" left] &
            I \ar[d, "f_0" left] \\
            I' &
            E' \ar[r, "p'"] \ar[l, "s'" above] &
            B' \ar[r, "t"] &
            I'
        \end{tikzcd}
    \]
    where the middle square is cartesian (i.e. is a pullback square). If $P$
    and $P'$ are both polynomial functors over $I$, then a morphism from $P$ to
    $P'$ \emph{over $I$} is a commutative diagram as above, but where $f_0$ is
    required to be the identity. Let $\PolyEnd$\index{$\PolyEnd$} denote the
    category of polynomial functors and morphisms of polynomial functors, and
    $\PolyEnd (I)$\index{$\PolyEnd (I)$} the category of polynomial functors
    over $I$ and morphisms of polynomial functors over $I$.
\end{definition}

\begin{remark}
    [Polynomial functors really are functors!]
    \label{rem:polynomial-functor-functor}
    The term ``polynomial (endo)functor'' is due to the association of $P$ to
    the composite endofunctor
    \[
        \Set / I
        \stackrel{s^*}{\longhookrightarrow} \Set / E
        \stackrel{p_*}{\longhookrightarrow} \Set / B
        \stackrel{t_!}{\longhookrightarrow} \Set / I
    \]
    where we have denoted $a_!$ and $a_*$ the left and right adjoints to the
    pullback functor $a^*$ along a map of sets $a$. Explicitly, for $(X_i \mid
    i \in I) \in \Set / I$, $P(X)$ is given by the ``polynomial''
    \begin{equation}
        \label{eq:polynomial-functor:eval}
        P (X) = \left( \sum_{b \in B(j)} \prod_{e \in E (b)} X_{s(e)} \mid j \in I \right) ,
    \end{equation}
    where $B (i) \eqdef t^{-1} (i)$ and $E (b) = p^{-1} (b)$. Visually,
    elements of $P (X)_j$ are nodes $b \in B$ such that $t b = j$, and whose
    inputs are decorated by elements of $(X_i \mid i \in I)$ in a manner
    compatible with their colours. Graphically, an element of $PX_i$ can be
    represented as
    \[
        \tikzinput{polynomial-functors}{evaluation}
    \]
    with $b \in B$ such that $t(b) = i$, and $x_j \in X_{s_{e_j} b}$ for $1
    \leq j \leq k$. Moreover, the endofunctor $P: \Set / I \longrightarrow \Set / I$
    preserves connected limits: $s^*$ and $p_*$ preserve all limits (as right
    adjoints), and $t_!$ preserves and reflects connected limits.

    This construction extends to a fully faithful functor $\PolyEnd(I) \longrightarrow
    \Cart (\Set / I)$, the latter being the category of endofunctors of $\Set /
    I$ and cartesian natural transformations\footnote{We recall that a natural
    transformation is \emph{cartesian}\index{cartesian natural transformation}
    if all its naturality squares are cartesian.}. In fact, the image of this
    full embedding consists precisely of those endofunctors that preserve
    connected limits \cite[section 1.18]{Gambino2013}. The composition of
    endofunctors gives $\Cart (\Set / I)$ the structure of a monoidal category,
    and $\PolyEnd(I)$ is stable under this monoidal product \cite[proposition
    1.12]{Gambino2013}. The identity polynomial functor
    $\polynomialinline{I}{I}{I}{I}$ is associated to the identity endofunctor;
    thus $\PolyEnd(I)$ is a monoidal subcategory of $\Cart (\Set / I)$.
\end{remark}

\subsection{Trees}
\label{sec:trees}

\begin{definition}
    [Polynomial tree]
    \label{def:polynomial-tree}
    A polynomial functor $T$ given by
    \[
        \polynomialfunctor{T_0}{T_2}{T_1}{T_0}{s}{p}{t}
    \]
    is a \emph{(polynomial) tree}\index{polynomial tree} \cite[section
    1.0.3]{Kock2011} if
    \begin{enumerate}
        \item the sets $T_0$, $T_1$ and $T_2$ are finite (in particular, each
        node has finitely many inputs);
        \item the map $t$ is injective;
        \item the map $s$ is injective, and the complement of its image $T_0 -
        \im s$ has a single element, called the \emph{root}\index{root};
        \item let $T_0 = T_2 + \{ r \}$, with $r$ the root, and define the
        \emph{walk-to-root}\index{walk-to-root function} function $\sigma$ by
        $\sigma (r) = r$, and otherwise $\sigma (e) = t p (e)$; then we ask
        that for all $x \in T_0$, there exists $k \in \bbNN$ such that
        $\sigma^k (x) = r$.
    \end{enumerate}
    We call the colours of a tree its \emph{edges}\index{edge} and the inputs
    of a node the \emph{input edges}\index{input edge} of that node.

    Let $\Tree$\index{$\Tree$} be the full subcategory of $\PolyEnd$ whose
    objects are trees. Note that it is the category of \emph{symmetric} or
    \emph{non-planar} trees (the automorphism group of a tree is in general
    non-trivial) and that its morphisms correspond to inclusions of non-planar
    subtrees. An \emph{elementary tree}\index{elementary polynomial tree} is a
    tree with at most one node. Let $\elTree$\index{$\elTree$} be the full
    subcategory of $\Tree$ spanned by elementary trees.
\end{definition}

\begin{definition}
    [$P$-tree]
    \label{def:p-tree}
    For $P \in \PolyEnd$, the category $\tree P$\index{$\tree P$} of
    \emph{$P$-trees}\index{$P$-tree} is the slice $\Tree / P$. The fundamental
    difference between $\Tree$ and $\tree P$ is that the latter is always rigid
    i.e. it has no non-trivial automorphisms \cite[proposition
    1.2.3]{Kock2011}. In particular, this implies that $\PolyEnd$ does not have
    a terminal object.
\end{definition}

\begin{notation}
    \label{not:p-tree}
    Every $P$-tree $T \in \tree P$ corresponds to a morphism from a tree (which
    we shall denote by $\underlyingtree{T}$\index{$\underlyingtree{T}$}) to
    $P$, so that $T : \underlyingtree{T} \longrightarrow P$. We point out that
    $\underlyingtree{T}_1$ is the set of nodes of $\underlyingtree{T}$, while
    $T_1 : \underlyingtree{T}_1 \longrightarrow P_1$ is a decoration of the
    nodes of $\underlyingtree{T}$ by nodes of $P$, and likewise for edges.
\end{notation}

\begin{definition}
    [Category of elements]
    \label{def:category-of-elements}
    For $P \in \PolyEnd$, its \emph{category of elements}\index{category of
    elements of a polynomial functor}\footnote{Not to be confused with the
    category of elements of a presheaf over some category.} $\elt
    P$\index{$\elt P$|see {category of elements of a polynomial functor}} is
    the slice $\elTree / P$. Explicitly, for $P$ as in
    \cref{eq:polynomial-functor}, the set of objects of $\elt P$ is $I + B$,
    and for each $b \in B$, there is a morphism $\tgt : t(b) \longrightarrow
    b$\index{$\tgt$}, and a morphism $\src_e : s_e(b) \longrightarrow
    b$\index{$\src_e$} for each $e \in E (b)$. Remark that there is no
    non-trivial composition of arrows in $\elt P$.
\end{definition}

\begin{proposition}
    [{\cite[proposition 2.1.3]{Kock2011}}]
    \label{prop:category-of-elements-slice}
    There is an equivalence of categories $\Psh{\elt P} \simeq \PolyEnd/P$.
\end{proposition}
\begin{proof}
    For $X \in \Psh{\elt P}$, construct the following polynomial functor over $P$:
    \[
        \begin{tikzcd}
            \sum_{i \in I} X_i
                \ar[d] &
            E_X
                \pullbackcorner \ar[l] \ar[r] \ar[d] &
            \sum_{b \in B} X_b
                \ar[r] \ar[d] &
            \sum_{i \in I} X_i
                \ar[d] \\
            I &
            E
                \ar[l] \ar[r] &
            B
                \ar[r] & I ,
        \end{tikzcd}
    \]
    where $E_X \longrightarrow \sum_{i \in I}X_i$ is given by the maps
    $X_{\src_e} : X_b \longrightarrow X_{\src_e b}$, for $b \in B$ and $e \in E
    (b)$. In the other direction, note that the full inclusion $\elTree
    \longhookrightarrow \PolyEnd$ induces a full inclusion $\iota : \elt P
    \longhookrightarrow \PolyEnd / P$ whose nerve functor $\PolyEnd / P
    \longrightarrow \Psh{\elt P}$ maps $Q \in \PolyEnd / P$ to $\PolyEnd / P
    (\iota -, Q)$. The two constructions are easily seen to define the required
    equivalence of categories.
\end{proof}

\subsection{Addresses}
\label{sec:polynomial-tree:addresses}

\begin{definition}
    [Address]
    \label{def:address}
    Let $T \in \Tree$ be a polynomial tree and $\sigma$ be its walk-to-root
    function (\cref{def:polynomial-tree}). We define the
    \emph{address}\index{address} function $\addr$\index{$\addr$|see {address}}
    on edges inductively as follows:
    \begin{enumerate}
        \item if $r$ is the root edge, let $\addr r \eqdef []$,
        \item if $e \in T_0 - \{ r \}$ and if $\addr \sigma (e) = [x]$, define
        $\addr e \eqdef [x e]$.
    \end{enumerate}
    The address of a node $b \in T_1$ is defined as $\addr b \eqdef \addr t
    (b)$. Note that this function is injective since $t$ is. Let
    $T^\nodesymbol$\index{$T^\nodesymbol$|see {node address}} denote its image,
    the set of \emph{node addresses}\index{node address} of $T$, and let
    $T^\leafsymbol$\index{$T^\leafsymbol$|see {leaf address}} be the set of
    addresses of leaf edges\index{leaf address}, i.e. those not in the image of
    $t$.

    Assume now that $T : \underlyingtree{T} \longrightarrow P$ is a $P$-tree.
    If $b \in \underlyingtree{T}_1$ has address $\addr b = [p]$, write
    $\src_{[p]} T \eqdef T_1 (b)$. For convenience, we let $T^\nodesymbol
    \eqdef \underlyingtree{T}^\nodesymbol$, and $T^\leafsymbol \eqdef
    \underlyingtree{T}^\leafsymbol$.
\end{definition}

\begin{remark}
    The formalism of addresses is a useful bookkeeping syntax for the
    operations of grafting and substitution on trees. The syntax of addresses
    will extend to the category of opetopes and will allow us to give a precise
    description of the composition of morphisms in the category of opetopes
    (see \cref{def:o}) as well as certain constructions on opetopic
    sets.
\end{remark}

\begin{notation}
    \label{not:marked-tree}
    We denote by $\treewithleaf P$\index{$\treewithleaf P$} the set of
    $P$-trees with a marked leaf, i.e. endowed with the address of one of its
    leaves. Similarly, we denote by $\treewithnode P$\index{$\treewithnode P$}
    the set of $P$-trees with a marked node.
\end{notation}


\subsection{Grafting}

\begin{definition}
    [Elementary $P$-trees]
    \label{def:elementary-p-trees}
    Let $P$ be a polynomial endofunctor as in equation
    \eqref{eq:polynomial-functor}. For $i \in I$, define $\itree{i} \in \tr
    P$\index{$\itree{i}$} as having underlying tree
    \begin{equation}
        \label{eq:polynomial-functor:itree}
        \polynomialfunctor
            {\{i\}}{\emptyset}{\emptyset}{\{i \},}
            {}{}{}
    \end{equation}
    along with the obvious morphism to $P$, that which maps $i$ to $i \in I$.
    This corresponds to a tree with no nodes and a unique edge decorated by
    $i$. Define $\ytree{b} \in \tr P$\index{$\ytree{b}$}, the
    \emph{corolla}\index{corolla} at $b$, as having underlying tree
    \begin{equation}
        \label{eq:polynomial-functor:ytree}
        \polynomialfunctor
            {s(E(b)) + \{*\}}{E (b)}{\{b \}}{s(E(b)) + \{*\},}
            {s}{}{}
    \end{equation}
    where the right map sends $b$ to $*$, and where the morphism $\ytree{b}
    \longrightarrow P$ is the identity on $s (E (b)) \subseteq I$, maps $*$ to
    $t (b) \in I$, is the identity on $E (b) \subseteq E$, and maps $b$ to $b
    \in B$. This corresponds to a $P$-tree with a unique node, decorated by
    $b$. Observe that for $T \in \tr P$, giving a morphism $\itree{i}
    \longrightarrow T$ is equivalent to specifying the address $[p]$ of an edge
    of $T$ decorated by $i$. Likewise, morphisms of the form $\ytree{b}
    \longrightarrow T$ are in bijection with addresses of nodes of $T$
    decorated by $b$.
\end{definition}

\begin{definition}
    [Grafting]
    \label{def:grafting}
    For $S, T \in \tr P$, $[l] \in S^\leafsymbol$ such that the leaf of $S$ at
    $[l]$ and the root edge of $T$ are decorated by the same $i \in I$, define
    the \emph{grafting}\index{grafting} $S \graft_{[l]} T$\index{$S
    \graft_{[l]} T$|see {grafting}} of $S$ and $T$ on $[l]$ by the following
    pushout (in $\tree P$):
    \begin{equation}
        \label{eq:polynomial-functor:grafting}
        \pushoutdiagram
            {\itree{i}}{T}{S}{S \graft_{[l]} T .}
            {[]}{[l]}{}{}
    \end{equation}
    Note that if $S$ (resp. $T$) is a trivial tree, then $S \graft_{[l]} T = T$
    (resp. $S$). We assume, by convention, that the grafting operator $\circ$
    associates to the right.
\end{definition}

\begin{proposition}
    [{\cite[proposition 1.1.21]{Kock2011}}]
    \label{prop:polynomial-functor:trees-are-graftings}
    Every $P$-tree is either of the form $\itree{i}$, for some $i \in I$, or
    obtained by iterated graftings of corollas (i.e. $P$-trees of the form
    $\ytree{b}$ for $b \in B$).
\end{proposition}

\begin{notation}
    [Total grafting]
    \label{not:total-grafting}
    Take $T, U_1, \ldots, U_k \in \tr P$, where $T^\leafsymbol = \left\{ [l_1],
    \ldots, [l_k] \right\}$, and assume the grafting $T \graft_{[l_i]} U_i$ is
    defined for all $i$. Then the \emph{total grafting} will be denoted
    concisely by
    \begin{equation}
        \label{eq:big-grafting}
        T \biggraft_{[l_i]} U_i
        =
        ( \cdots (T \graft_{[l_1]} U_1) \graft_{[l_2]} U_2 \cdots) \graft_{[l_k]} U_k .
    \end{equation}\index{$T \biggraft_{[l_i]} U_i$}
    It is easy to see that the result does not depend on the order in which the
    graftings are performed.
\end{notation}

\begin{definition}
    [Substitution]
    \label{def:substitution}
    Let $[p] \in T^\nodesymbol$ and $b = \src_{[p]} T$. Then $T$ can be
    decomposed as
    \begin{equation}
        \label{eq:node-decomposition}
        T = A \graft_{[p]} \ytree{b} \biggraft_{[e_i]} B_i ,
    \end{equation}
    where $E (b) = \{e_1, \ldots, e_k\}$, and $A, B_1, \ldots, B_k \in \tree
    P$. For $U$ a $P$-tree with a bijection $\readdress : U^\leafsymbol
    \longrightarrow E (b)$ over $I$, we define the
    \emph{substitution}\index{substitution} $T \subst_{[p]} U$\index{$T
    \subst_{[p]} U$|see {substitution}} as
    \begin{equation}
        \label{eq:subst}
        T \subst_{[p]} U
        \eqdef
        A \graft_{[p]} U \biggraft_{\readdress^{-1} e_i} B_i .
    \end{equation}
    In other words, the node at address $[p]$ in $T$ has been replaced by $U$,
    and the map $\readdress$ provided ``rewiring instructions'' to connect the
    leaves of $U$ to the rest of $T$.
\end{definition}


\subsection{Polynomial monads}
\label{sec:polynomial:monads}

\begin{definition}
    [Polynomial monad]
    \label{def:polynomial-monad}
    A \emph{polynomial monad over $I$}\index{polynomial monad} is a monoid in
    $\PolyEnd(I)$. Note that a polynomial monad over $I$ is thus necessarily a
    cartesian monad on $\Set/I$.\footnote{We recall that a monad is
    \emph{cartesian}\index{cartesian monad} if its endofunctor preserves
    pullbacks and its unit and multiplication are cartesian natural
    transformations.} Let $\PolyMnd(I)$\index{$\PolyMnd(I)$} be the category of
    monoids in $\PolyEnd(I)$. That is, $\PolyMnd(I)$ is the category of
    polynomial monads over $I$ and morphisms of polynomial functors over $I$
    that are also monad morphisms.
\end{definition}

\begin{definition}
    [$(-)^\star$ construction]
    \label{def:star-construction}
    Given a polynomial endofunctor $P$ as in equation
    \eqref{eq:polynomial-functor}, we define a new polynomial endofunctor
    $P^\star$\index{$P^\star$} as
    \begin{equation}
        \label{eq:free-monad}
        \polynomialfunctor
            {I}{\treewithleaf P}{\tree P}{I}
            {s}{p}{t}
    \end{equation}
    where $s$ maps a $P$-tree with a marked leaf to the decoration of that
    leaf, $p$ forgets the marking, and $t$ maps a tree to the decoration of its
    root. Remark that for $T \in \tree P$ we have $p^{-1} T = T^\leafsymbol$.
\end{definition}

\begin{theorem}
    [{\cite[section 1.2.7]{Kock2011}, \cite[sections 2.7 to 2.9]{Kock2010}}]
    \label{th:polynomial-monads-star-algebras}
    The polynomial functor $P^\star$ has a canonical structure of a polynomial
    monad. Furthermore, the functor $(-)^\star$ is left adjoint to the
    forgetful functor $\PolyMnd (I) \longrightarrow \PolyEnd (I)$, and the
    adjunction is monadic.
\end{theorem}

\begin{definition}
    [Readdressing function]
    \label{def:readdressing}
    We abuse notation slightly by letting $(-)^\star$ denote the associated
    monad on $\PolyEnd (I)$. Let $M$ be a polynomial monad as on the left
    below. Buy \cref{th:polynomial-monads-star-algebras}, $M$ is a
    $(-)^\star$-algebra, and we will write its structure map $M^\star
    \longrightarrow M$ as on the right:
    \begin{equation}
        \label{eq:polynomial-monad-structure-map}
        \polynomialfunctor{I}{E}{B}{I,}{s}{p}{t}
        \qqquad
        \begin{tikzcd}
            I
                \ar[d, equal] &
            \treewithleaf M
                \pullbackcorner
                \ar[l] \ar[d, "\readdress"] \ar[r] &
            \tree M
                \ar[r] \ar[d, "\tgt"] &
            I
                \ar[d, equal]
            \\
            I &
            E
                \ar[l] \ar[r] &
            B \ar[r] &
            I.
        \end{tikzcd}
    \end{equation}
    where $r T$ is the decoration of the root edge of a tree $T \in \tree M$.
    We call $\readdress_T : T^\leafsymbol \stackrel{\cong}{\longrightarrow} E
    (r T)$\index{$\readdress$|see {readdressing}} the
    \emph{readdressing}\index{readdressing} function of $T$, and $\tgt T \in B$
    is called the \emph{target} of $T$. If we think of an element $b \in B$ as
    the corolla $\ytree{b}$, then the target map $\tgt$ ``contracts'' a tree to
    a corolla, and since the middle square is a pullback, the number of leaves
    is preserved. The map $\readdress_T$ establishes a coherent correspondence
    between the set $T^\leafsymbol$ of leaf addresses of a tree $T$ and the
    elements of $E(T)$.
\end{definition}


\subsection{The Baez--Dolan construction}
\label{sec:polynomial:bd+}

\begin{definition}
    [Baez--Dolan $(-)^+$ construction]
    \label{def:baez-dolan-construction}
    Let $M$ be a polynomial monad as in \cref{eq:polynomial-functor}, and
    define its \emph{Baez--Dolan construction}\index{Baez--Dolan construction}
    $M^+$\index{$M^+$|see {Baez--Dolan construction}} to be
    \begin{equation}
        \label{eq:polynomial-functor:+}
        \polynomialfunctor{B}{\treewithnode M}{\tree M}{B}{\src}{p}{\tgt}
    \end{equation}
    where $\src$ maps an $M$-tree with a marked node to the label of that node,
    $p$ forgets the marking, and $\tgt$ is the target map. If $T \in \tree M$,
    remark that $p^{-1} T = T^\nodesymbol$ is the set of node addresses of $T$.
    If $[p] \in T^\nodesymbol$, then $\src [p] \eqdef \src_{[p]} T$.
\end{definition}

\begin{theorem} [{\cite[section 3.2]{Kock2010}}]
    \label{th:polynomial-functor:+:is-monad}
    The polynomial functor $M^+$ has a canonical structure of a polynomial
    monad.
\end{theorem}

\begin{remark}
    The $(-)^+$ construction is an endofunctor on $\PolyMnd$ whose definition
    is motivated as follows. If we begin with a polynomial monad $M$, then the
    colours of $M^+$ are the operations of $M$. The operations of $M^+$, along
    with their output colour, are given by the monad multiplication of $M$:
    they are the \emph{relations} of $M$, i.e. the reductions of trees of $M$
    to operations of $M$. The monad multiplication on $M^+$ is given as
    follows: the reduction of a tree of $M^+$ to an operation of $M^+$ (which
    is a tree of $M$) is obtained by \emph{substituting} trees of $M$ into
    nodes of trees of $M$.
\end{remark}

Let $M$ be a finitary (i.e. the fibers of $p$ below are finite, or
equivalently, $p_*$, and therefore $M = t_!p_*s^*$, preserves filtered
colimits) polynomial monad whose underlying polynomial functor is
\[
    \polynomialfunctor I E B {I.} s p t
\]
The Baez--Dolan construction gives the polynomial monad $M^+$ whose
underlying polynomial functor is
\[
    \polynomialfunctor B {\treewithnode M} {\tree M} {B.} \src p {\tgt}
\]
Recall also from \cref{th:polynomial-monads-star-algebras} that the category
$\PolyMnd (I)$ is the category of $(-)^\star$-algebras. The following fact is
analogous to \cref{prop:category-of-elements-slice} and is at the heart of the
Baez--Dolan construction (indeed, it is even the original \emph{definition} of
the construction, see \cite[definition 15]{Baez1998}).

\begin{proposition}
	\label{prop:algebraic-realization:equivalence-P+-algebras}
    For $M$ a polynomial monad, there is an equivalence of categories $\Alg
    (M^+) \simeq \PolyMnd(I) / M$.
\end{proposition}
\begin{proof}
    Given a $M^+$-algebra $M^+X \xrightarrow{x} X$ in $\Set/B$, define $\Phi X \in
    \PolyEnd(I) / M$ as
    \[
        \begin{tikzcd}
            I
                \ar[d, equal] &
            E_X \pullbackcorner
                \ar[l] \ar[d] \ar[r] &
            X
                \ar[r] \ar[d] &
            I
                \ar[d, equal]
            \\
            I &
            E
                \ar[l] \ar[r] &
            B \ar[r] &
            I.
        \end{tikzcd}
    \]
    There is an evident bijection $\tree \Phi X \cong M^+X$ in $\Set/I$, and
    the structure map $x$ extends by pullback along $E_X \longrightarrow X$ to a map $(\Phi
    X)^\star \longrightarrow \Phi X$ in $\PolyEnd(I)$. It is easy to verify that this
    determines a $(-)^\star$-algebra structure on $\Phi X$, and that the map
    $\Phi X \longrightarrow M$ in $\PolyEnd(I)$ is a morphism of $(-)^\star$-algebras.
    Conversely, given an $N \in \PolyMnd(I)/M$ whose underlying polynomial
    functor is
    \[
        \polynomialfunctor {I} {E'} {B'} {I,} {}{}{}
    \]
    then the bijection $\tree N \cong M^+ B'$ in $\Set/I$ and the
    $(-)^\star$-algebra map $N^\star \longrightarrow N$ provide a map $M^+ B' \xrightarrow{\Psi N}
    B'$ in $\Set/I$. It is easy to verify that $\Psi N$ is the structure map of
    a $M^+$-algebra and that the constructions $\Phi$ and $\Psi$ are functorial
    and mutually inverse.
\end{proof}

\begin{remark}
    The previous result provides an equivalence between $\PolyMnd(I)/M$ and the
    category of $M^+$-algebras. A ``coloured'' version of this result can be
    (informally) stated as follows: for $\Alg^{\mathrm{col}} (M^+)$ a suitable
    category of coloured $M^+$-algebras, there is an equivalence
    $\Alg^{\mathrm{col}}(M^+) \simeq \PolyMnd/M$, where $\PolyMnd$ is the
    category of all polynomial monads (for all $I$ in $\Set$).
\end{remark}

\section{Opetopes}
\label{sec:opetopes}

In this section, we use the formalism of polynomial functors and polynomial
monads of \cref{sec:polynomial-functors-and-monads} to define opetopes and
morphisms between them. This gives us a category $\bbOO$ of opetopes and a
category $\Psh\bbOO$ of opetopic sets. Our construction of opetopes is
precisely that of \cite{Kock2010}, and by \cite[theorem 3.16]{Kock2010}, also
that of \cite{Leinster2004,leinster2001structures}, and by \cite[corollary
2.6]{Cheng2004a}, also that of \cite{Cheng2003}. As we will see, the category
$\bbOO$ is rigid, i.e. it has no non-trivial automorphisms (it is in fact a
\emph{direct} category).

\subsection{Polynomial definition of opetopes}
\label{sec:opetopes:definition}

\begin{definition}
    [The $\optPolyFun^n$ monad]
    \label{def:zn}
    Let $\optPolyFun^0$ be the identity polynomial monad on $\Set$, as depicted
    on the left below, and let $\optPolyFun^n \eqdef
    (\optPolyFun^{n-1})^+$\index{$\optPolyFun^n$}. Write $\optPolyFun^n$ as on
    right:
    \begin{equation}
        \label{eq:zn}
        \polynomialfunctor
            {\{* \}}{\{* \}}{\{* \}}{\{* \},}
            {}{}{}
        \qqquad
        \polynomialfunctor
            {\bbOO_n}{E_{n+1}}{\bbOO_{n+1}}{\bbOO_n .}
            {\src}{p}{\tgt}
    \end{equation}
\end{definition}

\begin{definition}
    [Opetope]
    \label{def:opetope}
    An \emph{$n$-dimensional opetope} (or simply \emph{$n$-opetope}) $\omega$
    is by definition an element of $\bbOO_n$, and we write $\dim \omega = n$.
    An opetope $\omega \in \bbOO_n$ with $n \geq 2$ is called
    \emph{degenerate}\index{degenerate opetope} if its underlying tree has no
    nodes (thus consists of a unique edge); it is \emph{non
    degenerate}\index{non degenerate opetope} otherwise.

    Following \eqref{eq:polynomial-monad-structure-map}, for $\omega \in
    \bbOO_{n+2}$, the structure of polynomial monad $(\optPolyFun^n)^\star
    \longrightarrow \optPolyFun^n$ gives a bijection $\readdress_\omega :
    \omega^\leafsymbol \longrightarrow (\tgt \omega)^\nodesymbol$ between the
    leaves of $\omega$ and the nodes of $\tgt \omega$, preserving the
    decoration by $n$-opetopes.
\end{definition}

\begin{example}
    \begin{enumerate}
        \item The unique $0$-opetope is denoted $\optZero$\index{$\optZero$|see
        {point (opetope)}} and called the \emph{point}\index{point (opetope)}.
        \item The unique $1$-opetope is denoted $\optOne$\index{$\optOne$|see
        {arrow (opetope)}} and called the \emph{arrow}\index{arrow (opetope)}.
        \item If $n \geq 2$, then $\omega\in \bbOO_n$ is a
        $\optPolyFun^{n-2}$-tree, i.e. a tree whose nodes are labeled in
        $(n-1)$-opetopes, and edges are labeled in $(n-2)$-opetopes. In
        particular, $2$-opetopes are $\optPolyFun^0$-trees, i.e. linear trees,
        and thus in bijection with $\bbNN$. We will refer to them as
        \emph{opetopic integers}\index{opetopic integers}, and write
        $\optInt{n}$\index{$\optInt{n}$|see {opetopic integer}} for the
        $2$-opetope having exactly $n$ nodes.
    \end{enumerate}
\end{example}

\subsection{Higher addresses}

By definition, an opetope $\omega$ of dimension $n \geq 2$ is a
$\optPolyFun^{n-2}$-tree, and thus the formalism of tree addresses
(\cref{def:address}) can be applied to designate nodes of $\omega$, also called
its \emph{source faces}\index{source face} or simply
\emph{sources}\index{source|see {source face}}. In this section, we iterate
this formalism into the concept of \emph{higher dimensional
address}\index{higher dimensional address}\index{higher address|see {higher
dimensional address}}, which turns out to be more convenient. This material is
largely taken from \cite[section 2.2.3]{Curien2018} and \cite[section
4]{HoThanh2018}.

\begin{definition}
    [Higher address]
    \label{def:higher-address}
    Start by defining the set $\bbAA_n$\index{$\bbAA_n$|see {higher dimensional
    address}} of \emph{$n$-addresses} as follows:
    \[
        \bbAA_0 = \left\{ * \right\} , \qqquad
        \bbAA_{n+1} = \operatorname{lists} \bbAA_n ,
    \]
    where $\operatorname{lists} X$ is the set of finite lists (or words) on the
    alphabet $X$.

    Explicitly, the unique $0$-address is $*$ (also written $[]$ by
    convention), while an $(n+1)$-address is a sequence of $n$-addresses. Such
    sequences are enclosed by brackets. Note that the address $[]$, associated
    to the empty word, is in $\bbAA_n$ for all $n \geq 0$. However, the
    surrounding context will almost always make the notation unambiguous.

    Here are examples of higher addresses:
    \[
        [] \in \bbAA_1, \qqquad
        [*** \,*] \in \bbAA_1, \qqquad
        [[][*][]] \in \bbAA_2, \qqquad
        [[[[*]]]] \in \bbAA_4 .
    \]
\end{definition}

For $\omega \in \bbOO$ an opetope, nodes of $\omega$ can be specified uniquely
using higher addresses, as we now show. Recall that $E_{n-1}$ is the set of
arities of $\optPolyFun^{n-2}$ (\cref{eq:zn}). In $\optPolyFun^0$, set $E_1
(\optOne) = \{ * \}$, so that the unique node address of $\optOne$ is $* \in
\bbAA_0$. For $n \geq 2$, recall that an opetope $\omega \in
\bbOO_n$ is a $\optPolyFun^{n-2}$-tree $\omega : \underlyingtree{\omega}
\longrightarrow \optPolyFun^{n-2}$ (\cref{def:p-tree}), and write
$\underlyingtree{\omega}$ as
\[
    \polynomialfunctor{I_\omega}{E_\omega}{B_\omega}{I_\omega .}{}{}{}
\]
A node $b \in B_\omega$ has an address $\addr b \in \operatorname{lists}
E_\omega$, which by $\omega_2 : E_\omega \longrightarrow E_{n-1}$ is mapped to
an element of $\operatorname{lists} E_{n-1}$. By induction, elements of
$E_{n-1}$ are $(n-2)$-addresses, whence $\omega_2 (\addr b) \in
\operatorname{list} \bbAA_{n-2} = \bbAA_{n-1}$. For the induction step,
elements of $E_n (\omega)$ are nodes of $\underlyingtree{\omega}$, which we
identify by their aforementioned $(n-1)$-addresses. Consequently, for all $n
\geq 1$ and $\omega \in \bbOO_n$, elements of $E_n (\omega) =
\omega^\nodesymbol$ can be seen as the set of $(n-1)$-addresses of the nodes of
$\omega$, and similarly, $\omega^\leafsymbol$ can be seen as the set of
$(n-1)$-addresses of edges of $\omega$.


We now identify the nodes of $\omega \in \bbOO_{n+2}$ with their addresses. In
particular, for $[p] \in \omega^\nodesymbol$ a node address of $\omega$, we
make use of the notation $\src_{[p]} \omega$ of \cref{sec:trees} to refer to
the decoration of the node at address $[p]$, which is an $(n+1)$-opetope. Let
$[l] = [p[q]] \in \bbAA_{n-1}$ be an address such that $[p] \in
\omega^\nodesymbol$ and $[q] \in \left( \src_{[p]} \omega \right)^\nodesymbol$.
Then as a shorthand, we write
\begin{equation}
	\label{eq:preliminaries:edg}
    \edg_{[l]} \omega \eqdef \src_{[q]} \src_{[p]} \omega .
\end{equation}\index{$\edg_{[l]}$}

\begin{example}
	\label{ex:higher-addresses:3}
    Consider the $2$-opetope on the left, called $\optInt{3}$:
    \[
        \tikzinput{opetope-2-graphical}{3}
        \qqquad\longleadsto\qqquad
        \tikzinput[.8]{opetope-2-tree}{3,3.address}
    \]
    Its underlying pasting diagram consists of $3$ arrows $\optOne$ grafted
    linearly. Since the only node address of $\optOne$ is $* \in \bbAA_0$, the
    underlying tree of $\optInt{3}$ can be depicted as on the right. On the
    left of that three are the decorations: nodes are decorated with $\optOne
    \in \bbOO_1$, while the edges are decorated with $\optZero \in \bbOO_0$.
    For each node in the tree, the set of input edges of that node is in
    bijective correspondence with the node addresses of the decorating opetope,
    and that address is written on the right of each edges. In this low
    dimensional example, those addresses can only be $*$. Finally, on the right
    of each node of the tree is its $1$-address, which is just a sequence of
    $0$-addresses giving ``walking instructions'' to get from the root to that
    node.

    The $2$-opetope $\optInt{3}$ can then be seen as a corolla in some
    $3$-opetope as follows:
    \[
        \tikzinput{opetope-2-tree}{3.+,3.+.address}
    \]
    As previously mentioned, the set of input edges is in bijective
    correspondence with the set of node addresses of $\optInt{3}$.Here is now
    an example of a $3$-opetope, with its annotated underlying tree on the
    right (the $2$-opetopes $\optInt{1}$ and $\optInt{2}$ are analogous to
    $\optInt{3}$):
    \[
        \tikzinput{opetope-3-graphical}{ex1}
        \qqquad\longleadsto\qqquad
        \tikzinput{opetope-3-tree}{ex1,ex1.address}
    \]
\end{example}

\subsection{The category of opetopes}
\label{sec:opetopes:o}

In this subsection, we define the category $\bbOO$ of planar opetopes
introduced in \cite{HoThanh2018}, following the work of \cite{Cheng2003}.

\begin{proposition}
    [Opetopic identities, {\cite[theorem 4.1]{HoThanh2018}}]
    \label{th:opetopic-identities}
    Let $\omega \in \bbOO_n$ with $n \geq 2$.
    \begin{enumerate}
        \item (Inner edge) For $[p[q]] \in \omega^\nodesymbol$ (forcing $\omega$
        to be non degenerate), we have $\tgt \src_{[p[q]]} \omega = \src_{[q]}
        \src_{[p]} \omega$.
        \item (Globularity 1) If $\omega$ is non degenerate, we have $\tgt \src_{[]}
        \omega = \tgt \tgt \omega$.
        \item (Globularity 2) If $\omega$ is non degenerate, and $[p[q]] \in
        \omega^\leafsymbol$, we have $\src_{[q]} \src_{[p]} \omega =
        \src_{\readdress_\omega [p[q]]} \tgt \omega$.
        \item (Degeneracy) If $\omega$ is degenerate, we have $\src_{[]} \tgt
        \omega = \tgt \tgt \omega$.
    \end{enumerate}
\end{proposition}

\begin{definition}
    [{\cite[section 4.2]{HoThanh2018}}]
    \label{def:o}
    With these identities in mind, we define the category $\bbOO$ of opetopes
    by generators and relations as follows.
    \begin{enumerate}
        \item (Objects) We set $\ob \bbOO = \sum_{n \in \bbNN} \bbOO_n$.
        \item (Generating morphisms) Let $\omega \in \bbOO_n$ with $n \geq 1$.
        We introduce a generator $\tgt \omega \xrightarrow{\tgt}
        \omega$\index{$\tgt$|see {target embedding}}, called the \emph{target
        embedding}\index{target embedding}. If $[p] \in \omega^\nodesymbol$,
        then we introduce a generator $\src_{[p]} \omega
        \xrightarrow{\src_{[p]}} \omega$\index{$\src_{[p]}$|see {source
        embedding}}, called a \emph{source embedding}\index{source embedding}.
        A \emph{face embedding}\index{face embedding} is either a source or the
        target embedding.
        \item (Relations) We impose 4 relations described by the following
        commutative squares, that are well defined thanks to
        \cref{th:opetopic-identities}. Let $\omega \in \bbOO_n$ with
        $n \geq 2$
        \begin{enumerate}
            \item \condition{Inner}\index{inner@\condition{Inner}} for $[p[q]]
            \in \omega^\nodesymbol$ (forcing $\omega$ to be non degenerate),
            the following square must commute:
            \[
                \squarediagram
                    {\src_{[q]} \src_{[p]} \omega}{\src_{[p]} \omega}
                        {\src_{[p[q]]}\omega}{\omega}
                    {\src_{[q]}}{\tgt}{\src_{[p]}}{\src_{[p[q]]}}
            \]
            \item \condition{Glob1}\index{glob1@\condition{Glob1}} if $\omega$
            is non degenerate, the following square must commute:
            \[
                \squarediagram
                    {\tgt \tgt \omega}{\tgt \omega}{\src_{[]} \omega}{\omega .}
                    {\tgt}{\tgt}{\tgt}{\src_{[]}}
            \]
            \item \condition{Glob2}\index{glob2@\condition{Glob2}} if $\omega$
            is non degenerate, and for $[p [q]] \in \omega^\leafsymbol$, the
            following square must commute:
            \[
                \diagramsize{2}{4}
                \squarediagram
                    {\src_{\readdress_\omega [p [q]]} \tgt \omega}{\tgt \omega}
                        {\src_{[p]} \omega}{\omega .}
                    {\src_{\readdress_\omega [p[q]]}}{\src_{[q]}}{\tgt}
                        {\src_{[p]}}
                \diagramsize{2}{3}
            \]
            \item \condition{Degen}\index{degen@\condition{Degen}} if $\omega$
            is degenerate, the following square must commute:
            \[
                \squarediagram
                    {\tgt \tgt \omega}{\tgt \omega}{\tgt \omega}{\omega .}
                    {\tgt}{\src_{[]}}{\tgt}{\tgt}
            \]
        \end{enumerate}
    \end{enumerate}
    See \cite{HoThanh2018} for a graphical explanation of those relations.
\end{definition}

\begin{notation}
    \label{not:o}
    For $n \in \bbNN$, we let $\bbOO_{\leq n}$ be the full subcategory of $\bbOO$
    spanned by opetopes of dimension at most $n$. The subcategories $\bbOO_{< n}$,
    $\bbOO_{\geq n}$, $\bbOO_{> n}$, and $\bbOO_{= n}$ are defined similarly. Note
    that the latter is simply the set $\bbOO_n$.
\end{notation}

\subsection{Opetopic sets}
\label{sec:opetopes:opetopic-sets}

Recall from \cref{sec:preliminaries-category-theory} that $\PshO$ is the
category of opetopic sets, i.e. $\Set$-valued presheaves over $\bbOO$. For $X
\in \PshO$ and $\omega \in \bbOO$, we will refer to the elements of the set
$X_\omega$ as the \emph{cells}\index{cell} of $X$ of \emph{shape}\emph{shape}
$\omega$.

\begin{definition}
    \label{def:spine}
    \begin{enumerate}
        \item The representable presheaf at $\omega \in \bbOO_n$ is denoted
        $O[\omega]$. Its cells are morphisms of $\bbOO$ of the form $\sfF :
        \psi \longrightarrow \omega$, for $\sfF$ a sequence of face embeddings,
        which we write $\sfF \omega \in O [\omega]_\psi$ for short. For
        instance, the cell of maximal dimension is simply $\omega$ (as the
        corresponding sequence of face embeddings is empty), its $(n-1)$-cells
        are $\left \{ \src_{[p]} \omega \mid [p] \in \omega^\nodesymbol \right
        \} \cup \left \{ \tgt \omega \right \}$, and there is no cell of
        dimension $>n$.

        \item The \emph{boundary}\index{boundary} $\partial O
        [\omega]$\index{$\partial O [\omega]$|see {boundary}} of $\omega$ is
        the maximal subpresheaf of $O [\omega]$ not containing the cell
        $\omega$. We write $\sfB_\omega : \partial O [\omega]
        \longhookrightarrow O [\omega]$\index{$\sfB_\omega$|see {boundary
        inclusion}} for the \emph{boundary inclusion}\index{boundary
        inclusion}. The set of boundary inclusions is denoted by
        $\sfBB$\index{$\sfBB$|see {boundary inclusion}}

        \item The \emph{spine}\index{spine} $S [\omega]$\index{$S [\omega]$|see
        {spine}} is the maximal subpresheaf of $\partial O [\omega]$ not
        containing the cell $\tgt \omega$, and we write $\sfS_\omega : S
        [\omega] \longhookrightarrow O [\omega]$\index{$\sfS_\omega$|see {spine
        inclusion}} for the \emph{spine inclusion}\index{spine inclusion}. The
        set of spine inclusions is denoted by $\sfSS$\index{$\sfSS$|see {spine
        inclusion}}.
    \end{enumerate}
\end{definition}

\begin{lemma}
    \label{lemma:opetopes-technical:boundary-pushout}
    For $\omega \in \bbOO$, with $\dim \omega \geq 1$ the following square is a
    pushout and a pullback\footnote{Recall that in a topos, the pushout of a
    monomorphism along any arrow is a monomorphism, and the pushout square is a
    pullback square. This property is sometimes called ``adhesivity'', and is a
    consequence of van Kampen-ness, or descent, for pushouts of
    monomorphisms.}, where all arrows are canonical inclusions:
    \[
        \squarediagram
            {\partial O [\tgt \omega]}{S [\omega]}{O [\tgt \omega]}
                {\partial O [\omega] .}
            {}{}{}{}
    \]
\end{lemma}

\begin{lemma}
	\label{lemma:opetopes-technical:spine-pushout}
    Let $n \geq 1$, $\nu \in \bbOO_n$, $[l] \in \nu^\leafsymbol$, and $\psi \in
    \bbOO_{n-1}$ be such that $\edg_{[l]} \nu = \tgt \psi$, so that the grafting
    $\nu \graft_{[l]} \ytree{\psi}$ is well-defined. Then the following square
    is a pushout:
    \[
        \squarediagram
            {O [\edg_{[l]} \nu]}{O [\psi]}{S [\nu]}
                {S [ \nu \graft_{[l]} \ytree{\psi}].}
            {\tgt}{\edg_{[l]}}{\src_{[l]}}{}
    \]
\end{lemma}

\begin{notation}
    Let $F : \bbOO \longrightarrow \hom \catCC$ be a function that maps
    opetopes to morphisms in some category $\catCC$, and $\sfMM$ the set of
    maps defined by $\sfMM \eqdef \left \{ F (\omega) \mid \omega \in \bbOO
    \right \}$. Then for $n \in \bbNN$, we define $\sfMM_{\geq n} \eqdef \left
    \{ F (\omega) \mid \omega \in \bbOO_{\geq n} \right \}$, and similarly for
    $\sfMM_{> n}$, $\sfMM_{\leq n}$, $\sfMM_{< n}$, and $\sfMM_{= n}$. For
    convenience, the latter is abbreviated $\sfMM_n$. If $m \leq n$, we also
    let $\sfMM_{m, n} = \sfMM_{\geq m} \cap \sfMM_{\leq n}$. By convention,
    $\sfMM_{\leq n} = \emptyset$ if $n < 0$. For example, $\sfSS_{\geq 2} =
    \left \{ \sfS_\omega \mid \omega \in \bbOO_{\geq 2} \right \}$, and
    $\sfSS_{n, n+1} = \sfSS_n \cup \sfSS_{n+1}$.
\end{notation}

\begin{definition}
    \label{def:a}
    Write $\sfOO \eqdef \left\{ \emptyset \longhookrightarrow O [\omega] \mid
    \omega \in \bbOO \right\}$\index{$\sfOO$} for the set of initial inclusions
    of the representables. Let
    \begin{equation}
        \label{eq:opetopic-sets:a}
        \sfAA_{k, n} \eqdef \sfOO_{< n - k} \cup \sfSS_{\geq n + 1} .
    \end{equation}\index{$\sfAA$}
\end{definition}

\begin{lemma}
    \label{lemma:opetopes-technical:h-lift}
    \begin{enumerate}
        \item Let $X \in \PshO$ such that $\sfSS_{n,n+1} \perp X$. Then
        $\sfBB_{n+1} \perp X$. In general, every morphism in $\sfBB_{\geq n+1}$
        is an $\sfSS_{\geq n}$-local isomorphism.
        \item Let $X \in \PshO$ such that $\left( \sfSS_{n, n+1} \cup
        \sfBB_{n+2} \right) \perp X$. Then $\sfSS_{n+2} \perp X$. In particular,
        if $(\sfSS_{n, n+1} \cup \sfBB_{\geq n+2}) \perp X$, then $\sfSS_{\geq
        n} \perp X$.
    \end{enumerate}
\end{lemma}
\begin{proof}
    \begin{enumerate}
        \item Let $\omega \in \bbOO_{n+1}$. By 3-for-2, since the composite
        $\sfS_\omega : S [\omega] \longhookrightarrow \partial O [\omega] \longhookrightarrow O [\omega]$
        is in $\sfSS_{n,n+1}$, it suffices to show that $(S [\omega] \longhookrightarrow
        \partial O [\omega]) \perp X$. Let $f : S [\omega] \longrightarrow X$ be a
        morphism. The existence of a lift $\partial O [\omega] \longrightarrow X$ follows from
        the existence of a lift $O [\omega] \longrightarrow X$. Next, given two lifts $g,h : \partial
        O [\omega] \longrightarrow X$ of $f$, by
        \cref{lemma:opetopes-technical:boundary-pushout}, it suffices to show that
        they coincide on $O [\tgt\omega]$ to show that they are equal. But since
        they coincide on $S [\omega]$,
        they coincide on $S [\tgt\omega]$, and hence on $O [\tgt \omega]$ since
        $\sfS_{\tgt\omega} \perp X$.

        \item Let $\omega \in \bbOO_{n+2}$ and $f : S [\omega] \longrightarrow
        X$. Then the restriction $f|_{S [\tgt \omega]}$ of $f$ to $S [\tgt
        \omega]$ extends to a unique $f|_{S [\tgt \omega]}^{O[\tgt \omega]}$.
        We now show that the following square commutes: \diagramsize{2}{4}
        \begin{equation}
            \label{eq:lemma:opetopes-technical:h-lift}
            \squarediagram
                {\partial O [\tgt \omega]}{S [\omega]}{O [\tgt \omega]}{X .}
                {}{}{f}{f|_{S [\tgt \omega]}^{O[\tgt \omega]}}
            \diagramsize{2}{3}
        \end{equation}
        Tautologically, we have $f|_{S [\tgt \omega]} = f|_{S [\tgt
        \omega]}^{O[\tgt \omega]}|_{S [\tgt \omega]}$, and in particular,
        $f|_{S [\tgt \tgt \omega]} = f|_{S [\tgt \omega]}^{O[\tgt \omega]}|_{S
        [\tgt \tgt \omega]}$. Since $\sfSS_{n-2} \perp X$, we have $f|_{O [\tgt
        \tgt \omega]} = f|_{S [\tgt \omega]}^{O[\tgt \omega]}|_{O [\tgt \tgt
        \omega]}$. Therefore, square \eqref{eq:lemma:opetopes-technical:h-lift}
        commutes, and by \cref{lemma:opetopes-technical:boundary-pushout}, $f$
        extends to a unique $f|^{\partial O [\omega]} : \partial O [\omega]
        \longrightarrow O [\omega]$ such that $f|^{\partial O [\omega]}|_{O
        [\tgt \omega]} = f|_{S [\tgt \omega]}^{O[\tgt \omega]}$. Since $\sfBB_n
        \perp X$, $f|^{\partial O [\omega]}$ extends to a unique morphism $O
        [\omega] \longrightarrow X$.
        \qedhere
    \end{enumerate}
\end{proof}

\begin{lemma}
    \label{lemma:opetopes-technical:spines}
    Let $n \in \bbNN$, and $\omega \in \bbOO_{n+2}$. Then the inclusion
    $S[\tgt\omega] \longhookrightarrow S[\omega]$ is a relative $\sfSS_{n+1}$-cell complex,
    i.e. a transfinite composition of pushouts of maps in $\sfSS_{n+1}$.
\end{lemma}
\begin{proof}
    We show that the morphism $S[\tgt\omega] \longhookrightarrow S[\omega]$ is a (finite)
    composite of pushouts of elements of $\sfSS_{n+1}$. Let $X_0 =
    S[\tgt\omega]$. At stage $k\geq 0$, let $J$ be the (necessarily finite) set
    of inclusions $j = (j_1, j_2)$ in $\PshO^{\rightarrow}$ as on the left
    \[
        \diagramarrows{}{c->}{}{}
        \squarediagram
            {S[\psi_j]}{X_k}{O[\psi_j]}{S[\omega] ,}
            {j_1}{\sfS_{\psi_j}}{}{j_2}
        \qqquad
        \diagramarrows{}{c->}{}{}
        \pushoutdiagram
            {\coprod_{j\in J} S[\psi_j]}
            {X_k}
            {\coprod_{j\in J} O[\psi_j]}
            {X_{k+1} .}
            {j_1}{\sfS_{\psi_j}}{}{j_2}
    \]
    where $\psi_j \in \bbOO_{n+1}$, such that $j_1$ does not factor through
    $X_l$ for any $l < k$. Define $X_{k+1}$ with the pushout on the right.
    There is a $k$, bounded by the height of the tree $\omega \in \tree
    \optPolyFun^n$, at which the sequence converges to $S[\tgt\omega] =
    X_0 \longhookrightarrow X_1 \longhookrightarrow \ldots \longhookrightarrow X_k = S [\omega]$.
\end{proof}

\begin{corollary}
    \label{lemma:comparison:tgt-spine-local}
    \begin{enumerate}
        \item Let $n \in \bbNN$, and $\omega \in \bbOO_{n+2}$. Then the target
        embedding $\tgt \omega \longrightarrow \omega$ of $\omega$ is an
        $\sfSS_{n+1, n+2}$-local isomorphism.
        \item Let $k, n \in \bbNN$, and $\omega \in \bbOO_{\geq n+2}$. Then the
        target embedding $\tgt \omega \longrightarrow \omega$ is an
        $\sfAA_{k, n}$-local isomorphism.
    \end{enumerate}
\end{corollary}
\begin{proof}
    \begin{enumerate}
        \item In the square below
        \[
            \squarediagram {S [\tgt \omega]} {O [\tgt \omega]} {S
              [\omega]} {O [\omega]} {\sfS_{\tgt \omega}} {r} {\tgt}
            {\sfS_\omega}
        \]
        the map $r$ is an $\sfSS_{n+1}$-local isomorphism by
        \cref{lemma:opetopes-technical:spines}, and the horizontal maps are in
        $\sfSS_{n+1, n+2}$. The result follows by 3-for-2.
        \item Let $\omega \in \bbOO_m$. Since $\sfSS_{m-1, m+1} \subset
        \sfAA_{k,n}$ by definition, this follows from the previous point.
        \qedhere
    \end{enumerate}
\end{proof}

\begin{corollary}
    \label{coroll:y-i:spine-local}
    Let $\psi \in \bbOO_n$.
    \begin{enumerate}
        \item $\tgt \tgt = \src_{[]} \tgt : \psi \longrightarrow \itree{\psi}$
        is in $\sfSS_{n+2}$.
        \item The morphisms $\src_{[]}, \tgt : \psi \longrightarrow
        \ytree{\psi}$ are $\sfSS_{n+1, n+2}$-local isomorphisms.
    \end{enumerate}
\end{corollary}
\begin{proof}
    \begin{enumerate}
    \item The embedding $\tgt \tgt = \src_{[]} \tgt : \psi \longrightarrow
    \itree{\psi}$ is precisely the spine inclusion $\sfS_{\itree{\phi}}$ of the
    degenerate $(n+2)$-opetope $\itree{\psi}$.
    \item The source embedding $\src_{[]} : \psi \longrightarrow \ytree{\psi}$
    is precisely the spine inclusion $\sfS_{\ytree{\psi}}$ of the
    $(n+1)$-opetope $\ytree{\psi}$. The target embedding $\tgt : \psi \longrightarrow
    \ytree{\psi}$ is the morphism $\tgt : \tgt\tgt\itree{\psi} \longrightarrow
    \tgt\itree{\psi}$ and is the vertical arrow in the diagram below.
    \[
        \triangleDLdiagram
            {\psi = S [\itree{\psi}]}{\ytree{\psi} = \tgt \itree{\psi}}
                {\itree{\psi}.}
            {\tgt}{\sfS_{\itree{\psi}}}{\tgt}
    \]
    The horizontal arrow is an $\sfSS_{n+1, n+2}$-local isomorphism by
    \cref{lemma:comparison:tgt-spine-local} and the diagonal arrow is in
    $\sfSS_{n+2}$ by point (1). The result follows by 3-for-2.
    \qedhere
    \end{enumerate}
\end{proof}

\subsection{Extensions}

\begin{reminder}
    Recall that a functor between small categories $u : \catAA \longrightarrow \catBB$
    induces a restriction $u^* : \Psh\catBB \longrightarrow \Psh\catAA$ that
    admits both adjoints $u_! \dashv u^* \dashv u_*$, given by pointwise left
    and right Kan extensions.
\end{reminder}

\begin{notation}
    \label{not:subcategories-of-o}
    Let $m \in \bbNN$ and $n \in \bbNN \cup \{ \infty \}$ be such that $m \leq
    n$, and let $\bbOO_{m, n}$ be the full subcategory of $\bbOO$ spanned by
    opetopes $\omega$ such that $m \leq \dim \omega \leq n$. For instance,
    $\bbOO_{m, \infty} = \bbOO_{\geq m}$.
\end{notation}

\begin{definition}
    [Truncation]
    \label{def:truncation}
    The inclusion $\iota^{\geq m} : \bbOO_{m, n} \longrightarrow \bbOO_{\geq
    m}$ induces a restriction functor $(-)_{m, n} : \Psh{\bbOO_{\geq m}}
    \longrightarrow \Psh{\bbOO_{m, n}}$\index{$(-)_{m, n}$|see {truncation}},
    called \emph{truncation}\index{truncation}, that have both a left adjoint
    $\iota^{\geq m}_!$ and a right adjoint $\iota^{\geq m}_*$. Explicitly, for
    $X \in \Psh{\bbOO_{m, n}}$, the presheaf $\iota^{\geq m}_! X$ is the
    ``extension by $0$'', i.e. $(\iota^{\geq m}_! X)_{m, n} = X$, and
    $(\iota^{\geq m}_! X)_\psi = \emptyset$ for all $\psi \in \bbOO_{> n}$. On
    the other hand, $\iota^{\geq m}_* X$ is the ``canonical extension'' of $X$
    intro a presheaf over $\bbOO_{\geq m}$: we have $(\iota^{\geq m}_* X)_{m,
    n} = X$, and $\sfBB_{> n} \perp \iota^{\geq m}_* X$, which uniquely
    determines $\iota^{\geq m}_* X$.

    Likewise, the inclusion $\iota^{\leq n} : \bbOO_{m, n} \longrightarrow
    \bbOO_{\leq n}$ induces a restriction functor $\Psh{\bbOO_{\leq n}}
    \longrightarrow \Psh{\bbOO_{m, n}}$, also denoted by $(-)_{m, n}$ and again
    called \emph{truncation}\index{truncation}, that have both a left adjoint
    $\iota^{\leq n}_!$ and a right adjoint $\iota^{\leq n}_*$. Explicitly, for
    $X \in \Psh{\bbOO_{m, n}}$, the presheaf $\iota^{\leq n}_! X$ is the
    ``canonical extension'' of $X$ intro a presheaf over $\bbOO_{\leq n}$:
    \[
        \iota^{\leq n}_! X = \colim_{O[\psi]_{m, n} \rightarrow X} O [\psi] .
    \]
    On the other hand, $\iota^{\leq n}_* X$ is the ``terminal extension'' of
    $X$ in that $(\iota^{\leq n}_* X)_{m, n} = X$, and $(\iota^{\leq n}_*
    X)_\psi$ is a singleton, for all $\psi \in \bbOO_{< n}$. Note that
    $\sfOO_{< m} \perp \iota^{\leq n}_* X$, and that it is uniquely determined
    by this property.

    For $n < \infty$, we write $(-)_{\leq n}$ for $(-)_{0, n} :
    \Psh{\bbOO_{\geq 0}} = \PshO \longrightarrow \Psh{\bbOO_{0, n}} =
    \Psh{\bbOO_{\leq n}}$, and let $(-)_{< n} = (-)_{\leq n-1}$ if $n \geq 0$.
    Similarly, we note $(-)_{m, n} : \Psh{\bbOO_{\leq \infty}} = \PshO
    \longrightarrow \Psh{\bbOO_{m, \infty}} = \Psh{\bbOO_{\geq m}}$ by
    $(-)_{\geq m}$, and let $(-)_{> m} = (-)_{\geq m+1}$.
\end{definition}

\begin{proposition}
    \label{prop:properties-of-iotas}
    \begin{enumerate}
        \item The functors $\iota^{\geq m}_!$, $\iota^{\geq m}_*$, $\iota^{\leq
        n}_!$, and $\iota^{\leq n}_*$ are fully faithful.
        \item A presheaf $X \in \Psh{\bbOO_{\geq m}}$ is in the essential image
        of $\iota^{\geq m}_!$ if and only if $X_{> n} = \emptyset$.
        \item A presheaf $X \in \Psh{\bbOO_{\geq m}}$ is in the essential image
        of $\iota^{\geq m}_*$ if and only if for all $\omega \in \bbOO_{> n}$ we
        have $(\sfB_\omega)_{\geq m} \perp X$.
        \item A presheaf $X \in \Psh{\bbOO_{\leq n}}$ is in the essential image
        of $\iota^{\leq n}_*$ if and only if for all $\omega \in \bbOO_{< m}$
        we have $(\sfO_\omega)_{\leq n} \perp X$, i.e. $X_\omega$ is a
        singleton.
    \end{enumerate}
\end{proposition}
\begin{proof}
    The first point follows from the fact that $\iota^{\geq m}$ and
    $\iota^{\leq n}$ are fully faithful, and \cite[exposé I, proposition
    5.6]{SGA4}. The rest is straightforward verifications.
\end{proof}

\begin{notation}
    To ease notations, we sometimes leave truncations implicit, e.g. point (3)
    of last proposition can be reworded as: a presheaf $X \in \Psh{\bbOO_{\geq
    m}}$ is in the essential image of $\iota^{\geq m}_*$ if and only if
    $\sfBB_{> n} \perp X$.
\end{notation}


\section{The opetopic nerve of opetopic algebras}
\label{sec:algebraic-realization}

\subsection{Opetopic algebras}

Let $k \leq n \in \bbNN$. Recall from \cref{not:subcategories-of-o} that
$\bbOO_{n-k,n} \longhookrightarrow \bbOO$ is the full subcategory of opetopes
of dimension at least $n-k$ and at most $n$. A \emph{$k$-coloured,
$n$-dimensional opetopic algebra}, or $(k,n)$-opetopic algebra, will be an
algebraic structure on a presheaf over $\bbOO_{n-k,n}$. Specifically, we
describe a monad on the category $\Psh{\bbOO_{n-k,n}}$, whose algebras are the
$(k,n)$-opetopic algebras. Such an algebra $X$ has ``operations'' (its cells of
dimension $n$) that can be ``composed'' in ways encoded by
$(n+1)$-opetopes\footnote{Recall that an $(n+1)$-opetope is precisely a pasting
scheme of $n$-opetopes.}. The operations of $X$ will be ``coloured'' by its
cells dimension $< n$, which determines which operations can be composed
together.

As we will see, the fact that the operations and relations of an
$(k,n)$-opetopic algebra are encoded by opetopes of dimension $>n$ results in
the category $\Oalg k n$ of $(k,n)$-opetopic algebras always having a canonical
full and faithful \emph{nerve functor} to the category $\PshO$ of opetopic sets
(\cref{th:colored-algebraic-localization}).

We now claim some examples. A classification of $(k,n)$-opetopic algebras is
given by \cref{prop:algebra-table}.

\begin{example}
    [Monoids and categories] Let $k=0$ and $n=1$. Then $\bbOO_{n-k,n} = \bbOO_1
    = \{* \}$, and $\Psh{\bbOO_{n-k,n}} = \Set$. The category of
    $(0,1)$-opetopic algebras is precisely the category of associative monoids.
    If $k=1$ instead, then $\bbOO_{n-k,n} = \bbOO_{0,1}$, and
    $\Psh{\bbOO_{n-k,n}} = \Graph$, the category of directed graphs. The
    category of $(1,1)$-opetopic algebras is precisely the category of small
    categories.
\end{example}

\begin{example}
    [Coloured and uncoloured planar operads] Let $k=0$ and $n=2$. Then
    $\bbOO_{n-k,n} = \bbOO_2 \cong \bbNN$, and $\Psh{\bbOO_2} \simeq
    \Set/\bbNN$. The category of $(0,2)$-opetopic algebras is precisely the
    category of planar, uncoloured $\Set$-operads. If $k = 1$ instead, then
    $\Psh{\bbOO_{1,2}}$ is the category of planar, coloured collections. The
    category of $(1,2)$-opetopic algebras is precisely the category of planar,
    coloured $\Set$-operads.
\end{example}

\begin{example}
    [Combinads] Let $k=0$ and $n=3$. Then $\bbOO_{n-k,n} = \bbOO_3$ is the set
    of planar finite trees, and $(0,3)$-opetopic algebras are exactly the
    \emph{combinads}\index{combinad} over the combinatorial pattern of
    non-symmetric trees, presented in \cite{Loday2012a}.
\end{example}

\subsection{Parametric right adjoint monads}
\label{sec:pra}

In preparation to the main results of this section, we survey elements of the
theory of parametric right adjoint (p.r.a.) monads, which will be essential to
the definition and description of $(k,n)$-opetopic algebras. A comprehensive
treatment of this theory can be found in \cite{Weber2007}.

\begin{definition}
    [Parametric right adjoint]
    \label{def:pra}
    Let $T : \catCC \longrightarrow \catDD$ be a functor, and let $\catCC$ have
    a terminal object $1$. Then $T$ factors as
    \[
        \catCC
        = \catCC / 1
        \stackrel{T_1}{\longrightarrow} \catDD / T1
        \longrightarrow \catDD .
    \]
    We say that $T$ is a \emph{parametric right adjoint}\index{parametric right
    adjoint} (abbreviated p.r.a.\index{p.r.a.|see {parametric right adjoint}})
    if $T_1$ has a left adjoint $E$.
\end{definition}

We immediately restrict ourselves to the case $\catCC = \catDD = \Psh\catAA$
for a small category $\catAA$. Then $T_1$ is the nerve of the restriction $E :
\catAA / T1 \longrightarrow \Psh\catAA$, and the usual formula for nerve
functors gives
\begin{equation}
    \label{eq:pra:cardinal-formula}
    T X_a = \sum_{x \in T 1_a} \Psh\catAA (E x, X)
\end{equation}
for $X \in \Psh\catAA$ and $a \in \catAA$. In fact, it is clear that the data
of the object $T1 \in \Psh\catAA$ and of the functor $E$ completely describe
(via \cref{eq:pra:cardinal-formula}) the functor $T$ up to unique isomorphism.

\begin{definition}
    [P.r.a monad]
    \label{def:pra-monad}
    A \emph{p.r.a. monad}\index{p.r.a. monad} is a monad whose endofunctor is a
    p.r.a. and whose unit and multiplication are cartesian natural
    transformations\footnote{P.r.a. monads on presheaf categories are a strict
    generalisation of polynomial monads, since an endofunctor $\Set/I \longrightarrow
    \Set/I$ is a polynomial functor iff it is a p.r.a., see \cite[example
    2.4]{Weber2007}.}.

    Assume now that $T : \Psh\catAA \longrightarrow \Psh\catAA$ is a p.r.a.
    monad. Define $\Theta_0$ to be the full subcategory of $\Psh\catAA$ spanned
    by the image of $E : \catAA / T1 \longrightarrow \Psh\catAA$. Objects of
    $\Theta_0$ are called \emph{$T$-cardinals}\index{$T$-cardinals|see
    {cardinal}}\index{cardinal}. By \cite[proposition 4.20]{Weber2007}, the
    Yoneda embedding $\catAA \longhookrightarrow \Psh\catAA$ factors as
    \[
        \catAA
        \stackrel{i}{\longhookrightarrow} \Theta_0
        \stackrel{i_0}{\longhookrightarrow} \Psh\catAA
    \]\index{$\Theta_0$}
    or in other words, representable presheaves are $T$-cardinals.
\end{definition}

Every p.r.a. monad on a presheaf category is an example of a \emph{monad with
arities}\index{monad with arities} \cite{Berger2012a}. The theory of monads
with arities provides a remarkable amount of information about the
free-forgetful adjunction $\Psh\catAA \adjunction \Alg(T)$ and about the
category of algebras $\Alg (T)$\index{$\Alg (T)$}. We summarise those results
that we will use below.

\begin{proposition}
    \label{prop:pra-density}
    The fully faithful functor $i_0 : \Theta_0 \longhookrightarrow \Psh\catAA$
    is dense, or equivalently, its associated nerve functor $N_0 : \Psh\catAA
    \longhookrightarrow \Psh{\Theta_0}$ is fully faithful.
\end{proposition}
\begin{proof}
    We denote the inclusion
    $\catAA \longhookrightarrow \Theta_0$ by $i$, and note that $N_0$ is isomorphic to $i_*
    : \Psh\catAA \longrightarrow \Psh{\Theta_0}$. Now, $i$ is fully faithful, and by
    \cite[exposé I, proposition 5.6]{SGA4}, this is equivalent to $i_*$ being
    fully faithful.
\end{proof}

\begin{corollary}
    \label{coroll:pra-lifting}
    Let $J_\catAA \eqdef \{ \epsilon_\theta : i_! i^* \theta \longrightarrow
    \theta \mid \theta \in \Theta_0 - \im i \}$, where $\epsilon_\theta$ is the
    counit at $\theta$. Then a presheaf $X \in \Psh{\Theta_0}$ is in the
    essential image of $N_0$ if and only if $J_\catAA \perp X$.
\end{corollary}
\begin{proof}
    By the formula $i_* Y = \Psh\catAA ( \Theta_0 (i-,-) , Y )$ we have the
    sequence of isomorphisms
    \begin{align*}
        (i_* Y)_\theta
        &= \Psh\catAA ( \Theta_0 (i-,\theta) , Y ) \\
        &\cong \Psh\catAA (i^* \theta , i^*i_* Y)
            & \text{since $i_*$ is fully faithful} \\
        &\cong \Psh{\Theta_0} (i_!i^* \theta , i_* Y)
            & \text{since } \iota_! \dashv i^* .
    \end{align*}
    It is easy to check that one direction of the previous isomorphism is
    pre-composition by $\epsilon_\theta$. Conversely, take $X \in
    \Psh{\Theta_0}$ such that $\epsilon_\theta \perp X$ for all $\theta \in
    \Theta_0$. Then
    \begin{align*}
        i^* i_* X_\theta &\cong \Psh{\Theta_0} (\theta, i^* i_* X) \\
        &\cong \Psh{\Theta_0} (i_! i^* \theta, X) \\
        &\cong \Psh{\Theta_0} (\theta, X)
            & \text{since } \epsilon_\theta \perp X \\
        &\cong X_\theta ,
    \end{align*}
    and thus $X \in \im i^*$. Thus a presheaf $X$ is isomorphic to one of the
    form $i_* Y$ if and only if we have $\epsilon_\theta \perp X$ for all
    $\theta \in \Theta_0$. But if $\theta \in \im i$, then the associated
    counit map is already an isomorphism, hence we can restrict to $J_\catAA$.
\end{proof}

\begin{notation}
    Let the ``identity-on-objects / fully faithful'' factorisation of the
    composite functor $F_T i_0 : \Theta_0 \longhookrightarrow \Psh\catAA \longrightarrow
    \Alg(T)$ be denoted
    \begin{equation}
        \label{eq:thetaT}
        \Theta_0
        \stackrel{t}{\longrightarrow} \Theta_T
        \stackrel{i_T}{\longhookrightarrow} \Alg (T)
    \end{equation}\index{$\Theta_T$}
\end{notation}

\begin{theorem}
    \label{th:pra-nerve-theorem}
    \begin{enumerate}
        \item The fully faithful functor $i_T : \Theta_T \longhookrightarrow
        \Alg(T)$ is dense, or equivalently, its associated nerve functor $N_T :
        \Alg(T) \longhookrightarrow \Psh{\Theta_T}$ is fully faithful.
        \item The following diagram is an exact adjoint square\footnote{There
        exists a natural isomorphism $N_0 U_T \cong t^* N_T$ whose mate $t_!
        N_0 \longrightarrow N_T F_T$ is invertible (satisfies the Beck-Chevalley
        condition).}.
        \[
            \begin{tikzcd}
                \Psh\catAA
                    \ar[d, hook, "N_0" left]
                    \ar[r, shift left = .4em, "F_T"] &
                \Alg(T)
                    \ar[d, hook, "N_T"]
                    \ar[l, shift left = .4em, "U_T", "\perp" above] \\
                \Psh{\Theta_0}
                    \ar[r, shift left = .4em, "t_!"] &
                \Psh{\Theta_T}
                    \ar[l, shift left = .4em, "t^*", "\perp" above]
            \end{tikzcd}
        \]
        In particular, both squares commute up to natural isomorphism.
        \item (Nerve theorem) Any $X \in \Psh{\Theta_T}$ is in the essential
        image of $N_T$ if and only if $t^*X$ is in the essential image of
        $N_0$.
    \end{enumerate}
\end{theorem}
\begin{proof}
    See \cite[theorem 4.10]{Weber2007}, and \cite[proposition 1.9]{Berger2012a}.
\end{proof}

\begin{corollary}
    \label{coroll:algebra-lifting}
    Let $J_T \eqdef t_! J_\catAA = \{ t_! \epsilon_\theta : t_! i_! i^* \theta
    \longrightarrow t_! \theta \mid \theta \in \Theta_0 - \im i \}$, where
    $\epsilon_\theta : i_! i^* \theta \longrightarrow \theta$ is the counit at
    $\theta$. Then $X \in \Psh{\Theta_T}$ is in the essential image of $N_T$ if
    and only if $J_T \perp X$.
\end{corollary}
\begin{proof}
    This follows from \cref{th:pra-nerve-theorem}
    point (3), and from \cref{coroll:pra-lifting}.
\end{proof}

\subsection{Coloured \texorpdfstring{$\optPolyFun^n$}{zn}-algebras}
\label{sec:colored-zn-algebraic-realization}

In this section, we extend the polynomial monad $\optPolyFun^n$ over $\Set /
\bbOO_n = \Psh{\bbOO_n}$ to a p.r.a. monad
$\optPolyFun^n$\index{$\optPolyFun^n$} over $\Psh{\bbOO_{n-k, n}}$, for $k \leq
n \in \bbNN$.

This new setup will encompass more known examples (see
\cref{prop:algebra-table}). For instance, recall that the polynomial monad
$\optPolyFun^2$ on $\Set / \bbOO_2 \cong \Set / \bbNN$ is exactly the monad of
planar operads. The extension of $\optPolyFun^2$ will retrieve \emph{coloured}
planar operads as algebras. Similarly, the polynomial monad $\optPolyFun^1$ on
$\Set$ is the free-monoid monad, which we would like to vary to obtain
``coloured monoids'', i.e. small categories.

Let $k \leq n \in \bbNN$. Let us define a p.r.a. endofunctor $\optPolyFun^n$ on
the category $\Psh{\bbOO_{n-k, n}}$, that will restrict to the polynomial monad
$\optPolyFun^n : \Psh{\bbOO_n} \longrightarrow \Psh{\bbOO_n}$ in the case $k = 0$. Following \cref{sec:pra}, to define the p.r.a. endofunctor $\optPolyFun^n$ as
the composite
\[
    \Psh{\bbOO_{n-k, n}}
    \xrightarrow{\optPolyFun^n_1} \Psh{\bbOO_{n-k, n}}/\optPolyFun^n 1
    \simeq \Psh{\bbOO_{n-k, n} / \optPolyFun^n 1} \longrightarrow \Psh{\bbOO_{n-k, n}},
\]
up to unique isomorphism, it suffices to define its value $\optPolyFun^n 1$ on
the terminal presheaf, and to define a functor $E : \bbOO_{n-k, n} /
\optPolyFun^n 1 \longrightarrow \Psh{\bbOO_{n-k, n}}$.

\begin{definition}
    \label{def:colored-zn}
    Define $\optPolyFun^n 1$ as:
    \[
        (\optPolyFun^n1)_\psi
        \eqdef \{*\},
        \qqquad
        (\optPolyFun^n1)_\omega
        \eqdef \{\nu \in \bbOO_{n+1} \mid \tgt \nu = \omega \} .
    \]
    where $\psi \in \bbOO_{n-k, n-1}$ and $\omega \in \bbOO_n$. We define the
    functor $E : \bbOO_{n-k, n} / \optPolyFun^n 1 \longrightarrow
    \Psh{\bbOO_{n-k, n}}$ as follows. On objects, for $* \in
    (\optPolyFun^n1)_\psi$ and $\nu \in (\optPolyFun^n1)_\omega$, let
    \[
        E (*) \eqdef O [\psi],
        \qqquad
        E (\nu) \eqdef S [\nu] ,
    \]
    and on morphisms as the canonical inclusions. The functor $\optPolyFun^n_1
    : \Psh{\bbOO_{n-k, n}} \longrightarrow \Psh{\bbOO_{n-k, n} / \optPolyFun^n
    1}$ is defined as the right adjoint to the left Kan extension of $E$ along
    the Yoneda embedding. We now recover the endofunctor $\optPolyFun^n$
    explicitly using \cref{eq:pra:cardinal-formula}: for $\psi \in \bbOO_{n-k,
    n-1}$ we have $(\optPolyFun^n X)_\psi \cong X_\psi$, and for $\omega \in
    \bbOO_n$ we have
    \[
        (\optPolyFun^n X)_\omega
        \cong \sum_{\substack{\nu \in \bbOO_{n+1} \\ \tgt \nu = \omega}}
            \Psh{\bbOO_{n-k, n}} (S [\nu], X) .
    \]
\end{definition}

Note that $(\optPolyFun^n X)_\omega$ matches with the ``uncolored'' version of
$\optPolyFun_n$ of \cref{eq:zn}.

Recall from \cref{sec:pra} that a p.r.a. monad is a monad $M$ whose unit $\id
\longrightarrow M$ and multiplication $MM \longrightarrow M$ are cartesian, and
such that $M$ is a p.r.a. endofunctor. We now endow the p.r.a. endofunctor
$\optPolyFun^n$ with the structure of a p.r.a. monad over $\Psh{\bbOO_{n-k,
n}}$. We first specify the unit and multiplication $\eta_1 : 1 \longrightarrow
\optPolyFun^n 1$ and $\mu_1 : \optPolyFun^n \optPolyFun^n 1 \longrightarrow
\optPolyFun^n 1$ on the terminal object $1$, and extend them to cartesian
natural transformations (\cref{lemma:colored-zn:pra-monad-structure}). Next, we
check that the required monad identities hold for $1$
(\cref{lemma:colored-zn:pra-monad-structure:1}), which automatically gives us
the desired monad structure on $\optPolyFun^n$.

\begin{lemma}
    \label{lemma:zn-zn}
    The polynomial functor $\optPolyFun^n \optPolyFun^n : \Set / \bbOO_n
    \longrightarrow \Set / \bbOO_n$ is given by
    \[
        \polynomialfunctor
            {\bbOO_n}{E}{\bbOO_{n+2}^{(2)}}{\bbOO_n ,}
            {\edg}{}{\tgt \tgt}
    \]
    where
    \begin{enumerate}
        \item $\bbOO_{n+2}^{(2)}$ is the set of $(n+2)$ opetopes of height $2$,
        i.e. of the form
        \[
            \ytree{\nu} \biggraft_{[[p_i]]} \ytree{\nu_i} ,
        \]
        with $\nu, \nu_i \in \bbOO_{n+1}$ and $[p_i]$ ranging over a (possibly
        empty) subset of $\nu^\nodesymbol$,
        \item for $\xi \in \bbOO_{n+2}^{(2)}$, $E (\xi) = \xi^\leafsymbol$,
        \item for $\xi \in \bbOO_{n+2}^{(2)}$ and $[l] \in E (\xi) =
        \xi^\leafsymbol$, $\edg [l] = \edg_{[l]} \tgt$.
    \end{enumerate}
\end{lemma}

We now define $\eta_1$ and $\mu_1$.
\begin{enumerate}
    \item The presheaf $\optPolyFun^n 1 \in \Psh{\bbOO_{n-k, n}}$ is pointed in
    the following way: the morphism $\eta_1 : 1 \longrightarrow \optPolyFun^n
    1$ is the identity in dimension $< n$, and maps the unique element of
    $1_\omega$ to $\ytree{\omega}$ for $\omega \in \bbOO_n$.

    \item First, note that for every non-degenerate $\nu \in \bbOO_{n+1}$, a
    morphism $x : S [\nu] \longrightarrow \optPolyFun^n 1$ in $\Psh{\bbOO_{n-k,
    n}}$ and an element of $\optPolyFun^n \optPolyFun^n 1$, this corresponds to
    an opetope $\barnu \in \bbOO_{n+2}$ height $2$ (see \cref{lemma:zn-zn})
    such that $\src_{[]} \barnu = \nu$.

    Next, the map of presheaves $\mu_1 : \optPolyFun^n \optPolyFun^n 1
    \longrightarrow \optPolyFun^n 1$ is defined as the identity function on
    opetopes of dimension $< n$, and as the function
    \begin{align*}
        \mu_{1, \omega} :
            \sum_{\substack{\nu \in \bbOO_{n+1} \\ \tgt \nu = \omega}}
            \Psh{\bbOO_{n-k, n}} (S [\nu], \optPolyFun^n 1)
        &\longrightarrow
        \sum_{\substack{\nu \in \bbOO_{n+1} \\ \tgt \nu = \omega}}
            \Psh{\bbOO_{n-k, n}} (S [\nu], 1) \\
        (\nu, x) &\longmapsto (\tgt \barnu, *)
            & \text{if $\nu$ non-degen.} \\
        (\nu, *) &\longmapsto (\nu, *)
            & \text{if $\nu$ degen.}
    \end{align*}
    on opetopes $\omega \in \bbOO_n$. This is well-defined thanks to
    \condition{Glob1} (\cref{def:o}).
\end{enumerate}

\begin{lemma}
    \label{lemma:colored-zn:pra-monad-structure}
    For $X \in \Psh{\bbOO_{n-k, n}}$, the unique morphism $! : X
    \longrightarrow 1$ induces pullback squares
    \[
        \pullbackdiagram
            {X}{\optPolyFun^n X}{1}{\optPolyFun^n 1 ,}
            {\eta_X}{!}{\optPolyFun^n !}{\eta_1}
        \qqquad
        \pullbackdiagram
            {\optPolyFun^n \optPolyFun^n X}{\optPolyFun^n X}
                {\optPolyFun^n \optPolyFun^n 1}{\optPolyFun^n 1}
            {\mu_X}{\optPolyFun^n \optPolyFun^n !}{\optPolyFun^n !}{\mu_1}
    \]
\end{lemma}
\begin{proof}
    Straightforward verifications, see \cref{sec:omitted-proofs} for details.
\end{proof}

We have suggestively named the topmost arrows $\eta_X$ and $\mu_X$, since this
choice of pullback square for each $X$ gives cartesian natural transformations
$\eta : \id \longrightarrow \optPolyFun^n$ and $\mu : \optPolyFun^n
\optPolyFun^n \longrightarrow \optPolyFun^n$.

\begin{lemma}
    \label{lemma:colored-zn:pra-monad-structure:1}
    The following diagrams commute:
    \[
        \begin{tikzcd}
            \optPolyFun^n 1
                \ar[r, "\eta_{\optPolyFun^n 1}"]
                \ar[dr, equal] &
            \optPolyFun^n \optPolyFun^n 1
                \ar[d, "\mu_1"] &
            \optPolyFun^n 1
                \ar[l, "\optPolyFun^n \eta_1"']
                \ar[dl, equal] \\
            &
            \optPolyFun^n 1 , &
        \end{tikzcd}
        \qqquad
        \squarediagram
            {\optPolyFun^n \optPolyFun^n \optPolyFun^n 1}
                {\optPolyFun^n \optPolyFun^n 1}{\optPolyFun^n \optPolyFun^n 1}
                {\optPolyFun^n 1 .}
            {\optPolyFun^n \mu_1}{\mu_{\optPolyFun^n 1}}{\mu_1}{\mu_1}
    \]
\end{lemma}
\begin{proof}
    Straightforward computations. See \cref{sec:omitted-proofs} for details.
\end{proof}

\begin{proposition}
    The cartesial natural transformations $\mu$ and $\eta$ give $\optPolyFun^n$
    a structure of p.r.a. monad on $\Psh{\bbOO_{n-k, n}}$.
\end{proposition}
\begin{proof}
    This is a direct consequence of
    \cref{lemma:colored-zn:pra-monad-structure:1,lemma:colored-zn:pra-monad-structure}.
\end{proof}

Clearly, when $k = 0$, we recover the usual polynomial monad on $\Set /
\bbOO_n$.

\begin{definition}
    [Opetopic algebra]
    \label{def:opetopic-algebra}
    Let $\Oalg k n$\index{$\Oalg k n$|see {opetopic algebra}}, the category of
    \emph{$k$-coloured $n$-dimensional opetopic algebras}\index{opetopic
    algebra}, be the Eilenberg--Moore category of $\optPolyFun^n$ considered as
    a monad on $\Psh{\bbOO_{n-k, n}}$.
\end{definition}

\begin{proposition}
    \label{prop:algebra-table}
    Up to equivalence, and for small values of $k$ and $n$, the category $\Oalg
    k n$ is given by the following table:
    \begin{center}
        \begin{tabular} {c|cccc}
        $k \backslash n$ & $0$ & $1$ & $2$ & $3$ \\
        \hline
        $0$ & $\Set$ & $\Mon$ & $\Op$ & $\Comb_{\bbPP \bbTT}$ \\
        $1$ & & $\Cat$ & $\Op_{\mathrm{col}}$ & $\Oalg 1 3$ \\
        $2$ & & & $\Oalg 2 2$ & $\Oalg 2 3$ \\
        $3$ & & & & $\Oalg 3 3$
        \end{tabular}
    \end{center}
    where $\Set$ is the category of sets, $\Mon$ of monoids, $\Cat$ of small
    categories, $\Op$ of non coloured planar operads, $\Op_\mathrm{col}$ of
    coloured planar operads, and $\Comb_{\bbPP \bbTT}$\index{$\Comb_{\bbPP
    \bbTT}$} of combinads\index{combinad} over the combinatorial pattern of
    planar trees \cite{Loday2012a}. The lower half of the table is left empty
    since $\Oalg k n = \Oalg n n$ for $k \geq n$.
\end{proposition}

\subsection{\texorpdfstring{$\optPolyFun^n$}{Zn}-cardinals and opetopic shapes}
\label{sec:zn-cardinals-opetopic-shapes}

\begin{definition}
    [Opetopic shape]
    \label{def:opetopic-shape}
    Following \cref{sec:pra}, the category of $\optPolyFun^n$-cardinals
    (\cref{def:pra-monad}) is the full subcategory $i_0 : \Theta_0
    \longhookrightarrow \Psh{\bbOO_{n-k, n}}$ whose objects are the
    representables $\omega \in \bbOO_{n-k, n}$ and the spines $S[\nu]$, for
    $\nu \in \bbOO_{n+1}$. Analogous to \cref{eq:thetaT}, we denote the
    (identity-on-objects, fully faithful) factorisation of $\optPolyFun^n i_0 :
    \Theta_0 \longhookrightarrow \Psh{\bbOO_{n-k, n}} \longrightarrow \Oalg k
    n$ by
    \[
        \Theta_0 \stackrel{z}{\longrightarrow} \bbLambda_{k, n}
        \stackrel{u}{\longhookrightarrow} \Oalg k n .
    \]
    We call the category $\bbLambda_{k, n}$\index{$\bbLambda$|see {opetopic
    shape}} of free algebras on the $\optPolyFun^n$-cardinals the category of
    \emph{$(k,n)$-opetopic shapes}\index{opetopic shape}.
\end{definition}

For the rest of this section, we fix parameters $k \leq n \in \bbNN$ once and
for all, and suppress them in notation whenever unambiguous, e.g. $\bbLambda
\eqdef \bbLambda_{k, n}$, $\optPolyFun \eqdef \optPolyFun^n$, $\Alg \eqdef
\Oalg k n$.

\begin{definition}
    [Spine]
    \label{def:spine-lambda}
    For $\omega \in \bbOO_{n+1}$, let $\catSS_\omega \eqdef \bbOO_{n-k,n} /
    S[\omega]$\index{$\catSS_\omega$}. We denote the colimit, in $\PshLambda$,
    by
    \[
        S [h \omega]
        \eqdef
        \colim \left(
            \catSS_\omega
            \longrightarrow \bbOO_{n-k,n}
            \stackrel{h}{\longrightarrow}
            \bbLambda
            \longrightarrow \PshLambda
        \right) ,
    \]
    and call it the \emph{spine on $h \omega$}\index{spine}. Let $\sfS_{h
    \omega} : S[h \omega] \longhookrightarrow h \omega$ be the \emph{spine
    inclusion of $h \omega$}\index{spine inclusion|see {spine}}, and with a
    slight abuse of notations,
    \[
        \sfSS \eqdef \left\{
            \sfS_{h \omega} : S[ h \omega ] \longhookrightarrow h \omega
            \mid \omega \in \bbOO_{n+1}
        \right\} .
    \]
\end{definition}

The theory reviewed in \cref{sec:pra} has the following immediate consequences.

\begin{proposition}
    \label{prop:alg-shape-nerve}
    \begin{enumerate}
        \item The inclusion $u : \bbLambda \longhookrightarrow \Alg$ is dense, i.e., its
        associated nerve functor $N_u : \Alg \longrightarrow \PshLambda$ is fully
        faithful.
        \item Any $X \in \Psh \bbLambda$ is in the essential image of $N_u$ if
        and only if $\sfSS \perp X$.
        \item The reflective adjunction $u : \PshLambda \adjunction \Alg :
        N_u$ exhibits $\Alg \simeq \sfSS \inv \Psh \bbLambda$ as the
        localisation of $\PshLambda$ at the set of morphisms $\sfSS$.
    \end{enumerate}
\end{proposition}


\begin{example}
    The category $\bbLambda_{1, 1}$ is the category of simplices $\bbDelta$, and
    $\bbLambda_{2, 1}$ is the planar version of Moerdijk and Weiss's category of
    dendrices $\bbOmega$. Then \cref{prop:alg-shape-nerve} is the well-known
    fact that $\Cat$ and $\Op_{\mathrm{col}}$ have fully faithful nerve functors
    to $\Psh \bbDelta$ and to $\Psh \bbOmega$ respectively, exhibiting them as
    localisations of the respective presheaf categories at a set of \emph{spine}
    inclusions.\footnote{Sometimes called ``Grothendieck-Segal'' colimits.}
\end{example}

Our motivation stems from the following remarkable fact: there is a functor
$\dotH : \bbOO_{n-k,n+2} \longrightarrow \bbLambda$ (that we will define below), that is
neither full nor faithful (it is, however, \emph{surjective} on objects and on
morphisms), but is such that the composite functor $\Alg \longhookrightarrow \PshLambda
\xrightarrow{\dotH^*} \Psh{\bbOO_{n-k,n+2}}$ \emph{is} a fully faithful right adjoint,
and moreover exhibits $\Alg$ as the category of models of a projective sketch
on $\bbOO_{n-k,n+2}^\op$. In addition, the composite fully faithful functor
$\bbLambda \longhookrightarrow \Psh{\bbOO_{n-k,n+2}}$ is simply the nerve associated to
$\dotH$ (this says that $\dotH$ is dense), and allows us to view $\bbLambda$ as
a full subcategory of $(n-k,n+2)$-truncated opetopic sets, justifying the use
of the term ``opetopic shape''.

\begin{definition}
    \label{def:doth}
    The functor $\dotH: \bbOO_{n-k, n+2} \longrightarrow \bbLambda$ is defined on objects
    as follows:
    \begin{align*}
        \dotH : \bbOO_{n-k, n+2} &\longrightarrow \bbLambda \\
        \psi \in \bbOO_{\leq n} &\longmapsto \optPolyFun O [\psi] \\
        \omega \in \bbOO_{n+1} &\longmapsto \optPolyFun S [\omega] \\
        \xi \in \bbOO_{n+2} &\longmapsto \optPolyFun S [\tgt \xi] .
    \end{align*}\index{$\dotH$}

    On morphisms, $\dotH$ is the same as the free functor $\optPolyFun$ on
    $\bbOO_{n-k, n}$. Take $\omega \in \bbOO_{n+1}$ and $\xi \in \bbOO_{n+2}$.
    \begin{enumerate}
        \item Let $[p] \in \omega^\nodesymbol$, and $\dotH \left( \src_{[p]}
        \omega \xrightarrow{\src_{[p]}} \omega \right) = \optPolyFun \left( O
        [\src_{[p]} \omega] \xrightarrow{\src_{[p]}} S [\omega] \right)$.

        \item Let $\dotH \left( \tgt \omega \xrightarrow{\tgt} \omega \right) =
        \left(\optPolyFun O [\tgt \omega] \xrightarrow{\dotH \tgt} \optPolyFun
        S [\omega] \right)$ correspond to the cell $\id_{S [\omega]} \in
        \optPolyFun S [\omega]_{\tgt \omega}$ under the Yoneda embedding.

        \item Let $\dotH \left( \tgt \xi \xrightarrow{\tgt} \xi \right) =
        \left(\optPolyFun S [\tgt \xi] \xrightarrow{\dotH \tgt} \optPolyFun S
        [\tgt \xi] \right)$ be the identity map.

        \item Let $[p] \in \xi^\nodesymbol$. In order to define $\dotH
        \left(\src_{[p]} \xi \xrightarrow{\src_{[p]}} \xi \right) = \left( \optPolyFun
        S [\src_{[p]} \xi] \xrightarrow{\dotH \src_{[p]}} \optPolyFun S [\tgt \xi]
        \right)$, it is enough to provide a morphism $\dotH \src_{[p]} :
        S[\src_{[p]}\xi] \longrightarrow \optPolyFun S [\tgt \xi]$ in $\Psh{\bbOO_{n-k,
        n}}$, which we now construct.

        Using \cref{eq:node-decomposition}, $\xi$ decomposes as
        \[
            \xi = \zeta \graft_{[p]} \ytree{\src_{[p]} \xi}
                \biggraft_{[[q_i]]} \zeta_i ,
        \]
        where $[q_i]$ ranges over $(\src_{[p]} \xi)^\nodesymbol$. The leaves of
        $\zeta_i$ are therefore a subset of the leaves of $\xi$. Precisely, a
        leaf address $[r] \in \zeta_i^\leafsymbol$ corresponds to the leaf
        address $[p[q_i]r] \in \xi^\leafsymbol$, defining an inclusion $f_i : S
        [\tgt \zeta_i] \longhookrightarrow S [\tgt \xi]$ that maps the node
        $\readdress_{\zeta_i} [r] \in (\tgt \zeta_i)^\nodesymbol$ to
        $\readdress_\xi [p[q_i]r] \in (\tgt \xi)^\nodesymbol$.

        Note that by definition, each $f_i$ is an element of $\Psh{\bbOO_{n-k,
        n}} (S [\tgt \zeta_i], S [\tgt \xi]) \subseteq \optPolyFun S [\tgt
        \xi]_{\tgt \tgt \zeta_i}$, and since $\tgt \tgt \zeta_i = \tgt
        \src_{[]} \zeta_i = \src_{[q_i]} \src_{[p]} \xi$ (by \condition{Glob1}
        and \condition{Inner}), we have $f_i \in \optPolyFun S [\tgt
        \xi]_{\src_{[q_i]} \src_{[p]} \xi}$.

        Together, the $f_i$ assemble into the required morphism $\dotH
        \src_{[p]} : S [\src_{[p]} \xi] \longrightarrow \optPolyFun S [\tgt
        \xi]$, that maps the node $[q_i] \in (\src_{[p]} \xi)^\nodesymbol$ to
        $f_i$. So in conclusion, we have
        \begin{align*}
            \dotH \src_{[p]} :
                S [\src_{[p]} \xi]
                &\longrightarrow
                \optPolyFun S [\tgt \xi] \\
            (\dotH \src_{[p]}) ([q_i]) :
                S [\tgt \zeta_i]
                &\longrightarrow
                S [\tgt \xi] \\
            \readdress_{\zeta_i} [r]
                &\longmapsto
                \readdress_\xi [p [q_i] r] ,
        \end{align*}
        for $[q_i] \in (\src_{[p]} \xi)^\nodesymbol$ and $[r] \in
        \zeta_i^\leafsymbol$.
    \end{enumerate}
    This defines $\dotH$ on object and morphisms, and functoriality is
    straightforward.
\end{definition}

\begin{example}
    Consider the case $(k, n) = (1, 1)$, so that $\dotH : \bbOO_{0, 3}
    \longrightarrow \bbLambda_{1, 1} \cong \bbDelta$. In low dimensions, we
    have $\dotH_{1, 1} \optZero = [0]$, $\dotH_{1, 1} \optOne = [1]$, and
    $\dotH_{1, 1} \optInt{n} = [n]$, for $n \in \bbNN$. Consider now the
    following $3$-opetope $\xi$:
    \[
        \xi
        = \ytree{\optInt{3}} \graft_{[[*]]} \ytree{\optInt{2}} \graft_{[[**]]} \ytree{\optInt{1}}
        = \left( \tikzinput[.7]{opetope-3-graphical}{ex1} \right)
    \]
    Then $\dotH_{1, 1} \xi = \optPolyFun^1 S [\tgt \xi] = [4]$. This
    result should be inderstood as the poset of points of $\xi$ (represented as
    dots in the pasting diagram above) ordered by the topmost arrows of $\xi$.

    Take the face embedding $\src_{[]} : \optInt{3} \longrightarrow \xi$. Then
    $\dotH_{1, 1} \src_{[]}$ maps points $0$, $1$, $2$, $3$ of $\dotH_{1, 1}
    \optInt{3} = [3]$ to points $0$, $1$, $3$, $4$ of $\dotH_{1, 1} \xi$,
    respectively. In other words, it ``skips'' point $2$, which is exactly what
    the pasting diagram above depicts: the $[]$-source of $\xi$ does not touch
    point $2$. Likewise, the map $\dotH_{1, 1} \src_{[[**]]} : [1] = \dotH_{1,
    1} \optInt{1} \longrightarrow \dotH_{1, 1} \xi$ maps $0$, $1$ to $3$, $4$,
    respectively.

    Consider now the target embedding $\tgt : \optInt{4} \longrightarrow \xi$.
    Since the target face touches all the points of $\xi$ (this can be checked
    graphically, but more generally follows from \condition{Glob2}), $\dotH_{1,
    1} \tgt$ should be the identity map on $[4]$, which is precisely what the
    definition gives.
\end{example}

We spend the rest of this section proving various (rather technical) facts
about the functor $\dotH$, which will allow us to construct the opetopic nerve
functor (and to prove that it is fully faithful) in
\cref{subsection:opetopic-nerve-zn-algebras}.

\begin{definition}
    \label{def:diagrammatic-morphism}
    Let $\omega, \omega' \in \bbOO_{n+1}$. A morphism $f : \dotH \omega
    \longrightarrow \dotH \omega'$ in $\bbLambda$ is
    \emph{diagrammatic}\index{diagrammatic morphism} if there exists a $\xi \in
    \bbOO_{n+2}$ and a $[p] \in \xi^\nodesymbol$ such that $\src_{[p]} \xi =
    \omega$, $\tgt \xi = \omega'$, and $f = \dotH \left( \omega
    \xrightarrow{\src_{[p]}} \xi \right)$. This situation is summarised by the
    following diagram, called a \emph{diagram of $f$}\index{diagram}.
    \[
        \frac{
            \begin{tikzcd} [ampersand replacement = \&]
                \&
                \xi \\
                \omega \ar[ur, sloped, near end, "\src_{[p]}"] \&
                \omega' \ar[u, "\tgt"]
            \end{tikzcd}
        }{
            \begin{tikzcd} [ampersand replacement = \&]
                \dotH \omega \ar[r, "f"] \& \dotH \omega'
            \end{tikzcd}
        }
    \]
\end{definition}

\begin{example}
    Consider the case $(k, n) = (1, 1)$ again, and recall that $\bbLambda_{1,
    1} \cong \bbDelta$. Consider the map $f : [2] \longrightarrow [3]$ mapping
    $0$, $1$, and $2$ to $0$, $2$, and $3$, respectively (in other words, $f$
    is the \emph{1st simplicial coface map} \cite{Jardine2006}). Taking $\xi$
    as on the left, we obtain a diagram of $f$ on the right:
    \[
        \xi
        = \ytree{\optInt{2}} \graft_{[[*]]} \ytree{\optInt{2}}
        = \left( \tikzinput{opetope-3-graphical}{classic} \right) ,
        \qqquad
        \frac{
            \begin{tikzcd} [ampersand replacement = \&]
                \&
                \xi \\
                \optInt{2} \ar[ur, sloped, near end, "\src_{[]}"] \&
                \optInt{2} \ar[u, "\tgt"]
            \end{tikzcd}
        }{
            \begin{tikzcd} [ampersand replacement = \&]
                {} [2] \ar[r, "g"] \& {} [3]
            \end{tikzcd}
        }
    \]

    Consider now a non injective map $g : [2] \longrightarrow [1]$ mapping $0$,
    $1$, and $2$ to $0$, $1$, and $1$, respectively (in other words, $g$ is the
    \emph{1st simplicial codegeneracy map} \cite{Jardine2006}). Taking $\zeta$
    as on the left, we obtain a diagram of $g$ on the right:
    \[
        \zeta
        = \ytree{\optInt{2}} \graft_{[[*]]} \ytree{\optInt{0}}
        = \left( \tikzinput[.9]{opetope-3-graphical}{degen2} \right) ,
        \qqquad
        \frac{
            \begin{tikzcd} [ampersand replacement = \&]
                \&
                \zeta \\
                \optInt{2} \ar[ur, sloped, near end, "\src_{[[*]]}"] \&
                \optInt{1} \ar[u, "\tgt"]
            \end{tikzcd}
        }{
            \begin{tikzcd} [ampersand replacement = \&]
                {} [2] \ar[r, "g"] \& {} [1]
            \end{tikzcd}
        }
    \]

    On the one hand, \cref{lemma:diagramatic-composite-diagramatic} below
    states that the composite of diagrammatic morphisms remains diagrammatic,
    and on the other hand, those two examples seem to indicate that all
    simplicial cofaces and codegeneracies are diagrammatic. One might thus
    expect all morphisms of $\bbDelta$ to be in the image of $\dotH_{1, 1} :
    \bbOO_{0, 3} \longrightarrow \bbLambda_{1, 1} \cong \bbDelta$. This is
    true, and a more general statement is proved in \cref{coroll:h-surjective}.
\end{example}

\begin{lemma}
    \label{lemma:diagramatic-composite-diagramatic}
    If $f_1$ and $f_2$ are diagrammatic as on the left, the diagram on the right
    is well defined, and is a diagram of $f_2 f_1$.
    \[
        \frac{
            \begin{tikzcd} [ampersand replacement = \&]
                \& \xi_1 \& \xi_2 \\
                \omega_0 \ar[ur, sloped, near end, "\src_{[p_1]}"] \&
                \omega_1
                    \ar[u, "\tgt"]
                    \ar[ur, sloped, near end, "\src_{[p_2]}"] \&
                \omega_2 \ar[u, "\tgt"]
            \end{tikzcd}
        }{
            \begin{tikzcd} [ampersand replacement = \&]
                \dotH \omega_0 \ar[r, "f_1"] \&
                \dotH \omega_1 \ar[r, "f_2"] \&
                \dotH \omega_2 ,
            \end{tikzcd}
        }
        \qqquad
        \frac{
            \begin{tikzcd} [ampersand replacement = \&]
                \& \& \xi_2 \subst_{[p_2]} \xi_1 \\
                \omega_0 \ar[urr, sloped, near end, "\src_{[p_2 p_1]}"] \&
                \&
                \omega_2 \ar[u, "\tgt"]
            \end{tikzcd}
        }{
            \begin{tikzcd} [ampersand replacement = \&]
                \dotH \omega_0 \ar[rr, "f_2 f_1"] \& \& \dotH \omega_2
            \end{tikzcd}
        }
    \]
\end{lemma}
\begin{proof}
    It is a simple but lengthy matter of unfolding the definition of $\dotH$.
    See \cref{sec:omitted-proofs} for details.
\end{proof}

\begin{lemma}
    [Contraction associativity formula]
    \label{lemma:caf}
    Let $n \geq 2$, $\nu, \nu' \in \bbOO_n$, and $[l] \in \nu^\nodesymbol$ be
    such that $\edg_{[l]} \nu = \edg_{[]} \nu'$. In particular, the
    grafting $\nu \graft_{[l]} \nu'$ is well-defined, and by \condition{Glob1}
    and \condition{Glob2}, $\src_{\readdress_\nu [l]} \tgt \nu = \edg_{[l]} \nu
    = \edg_{[]} \nu' = \tgt \tgt \nu'$. We have
    \[
        \tgt (\nu \graft_{[l]} \nu')
        =
        (\tgt \nu) \subst_{\readdress_\nu [l]} (\tgt \nu') .
    \]
\end{lemma}
\begin{proof}
    This is a direct consequence of the the fact that $\optPolyFun^{n-2}$ is a
    polynomial monad, i.e. a $(-)^\star$-algebra.
\end{proof}

\begin{lemma}
    \label{lemma:elementary-embeddings-diagramatic}
    \begin{enumerate}
        \item Let $\omega \in \bbOO_{n+1}$, and $\psi = \tgt \omega$. Then the
        following is a diagram of $h \tgt : \dotH \psi \longrightarrow \dotH \omega$:
        \[
            \frac{
                \begin{tikzcd} [ampersand replacement = \&]
                    \& \xi \\
                    \ytree{\psi} \ar[ur, sloped, near end, "\src_{[]}"] \&
                    \omega \ar[u, "\tgt"]
                \end{tikzcd}
            }{
                \begin{tikzcd} [ampersand replacement = \&]
                    \dotH \psi \ar[r, "h \tgt"] \& \dotH \omega ,
                \end{tikzcd}
            }
            \qqquad
            \xi \eqdef \ytree{\ytree{\psi}} \graft_{[[]]} \ytree{\omega} .
        \]
        (Note that $\omega = \tgt \xi$ by \cref{lemma:caf})

        \item Let $\beta, \omega \in \bbOO_{n+1} = \tree \optPolyFun^{n-1}$,
        and $i : S [\beta] \longrightarrow S [\omega]$ a morphism of
        presheaves. Then $i$ corresponds to an inclusionb of
        $\optPolyFun^{n-1}$ trees $\beta \longhookrightarrow \omega$, mapping
        node at address $[q]$ to $[pq]$, where $[p] = i [] \in
        \omega^\nodesymbol$. Write $\omega = \barbeta \subst_{[p]} \beta$, for
        an adequate $\barbeta \in \bbOO_{n+1}$. Then following is a diagram of
        $\dotH i$:
        \[
            \frac{
            \begin{tikzcd} [ampersand replacement = \&]
                \& \xi \\
                \beta \ar[ur, sloped, near end, "\src_{[p]}"] \&
                \omega \ar[u, "\tgt"]
            \end{tikzcd}
            }{
            \begin{tikzcd} [ampersand replacement = \&]
                \dotH \beta \ar[r, "\dotH i"] \& \dotH \omega ,
            \end{tikzcd}
            }
            \qqquad
            \xi \eqdef \ytree{\barbeta} \graft_{[[p]]} \ytree{\beta} .
        \]
        (Note that $\omega = \tgt \xi$ by \cref{lemma:caf})
    \end{enumerate}
\end{lemma}
\begin{proof}
    It is a simple matter of unfolding the definition of $\dotH$.
    See \cref{sec:omitted-proofs} for details.
\end{proof}

\begin{lemma}
    [Diagrammatic lemma]
    \label{lemma:diagramatic-lemma}
    Let $\omega, \omega' \in \bbOO_{n+1}$ with $\omega$ non degenerate, and
    $f : \dotH \omega \longrightarrow \dotH \omega'$. Then $f$ is diagrammatic.
\end{lemma}
\begin{proof}
    [Proof (sketch, see \cref{sec:omitted-proofs} for details)] The idea is to
    proceed by induction on $\omega$. The case $\omega = \ytree{\psi}$ for some
    $\psi \in \bbOO_n$ is fairly simple. In the inductive case ($\omega = \nu
    \graft_{[l]} \ytree{\psi}$, for suitable $\nu$, $\psi$, and $[l]$), we
    essentially show that $f$ exhibits an inclusion of $\optPolyFun^n$-trees
    $\omega \longhookrightarrow \omega'$ by constructing a $(n+1)$-opetope
    $\baromega$ such that $\omega' = \baromega \subst_{[q]} \omega$. Thus by
    \cref{lemma:elementary-embeddings-diagramatic}, the following is a diagram
    of $\dotH f$:
    \[
        \frac{
        \begin{tikzcd} [ampersand replacement = \&]
            \& \xi \\
            \omega \ar[ur, sloped, near end, "\src_{[[q_1]]}"] \&
            \omega' \ar[u, "\tgt"]
        \end{tikzcd}
        }{
        \begin{tikzcd} [ampersand replacement = \&]
            \dotH \omega \ar[r, "f"] \& \dotH \omega' ,
        \end{tikzcd}
        }
        \qqquad
        \xi \eqdef \ytree{\baromega} \graft_{[[q_1]]} \ytree{\omega} .
    \]
\end{proof}

\subsection{The opetopic nerve functor}
\label{subsection:opetopic-nerve-zn-algebras}

This section is entirely devoted to constructing the \emph{opetopic nerve
functor} $N : \Alg \longhookrightarrow \Psh\bbOO$, which is a fully faithful
right adjoint and which exhibits $\Alg$ as the category of models of a
projective sketch on $\bbOO^\op$ (\cref{th:colored-algebraic-localization}).

Recall from \cref{coroll:algebra-lifting} that the reflective adjunction $\Psh
\bbLambda \adjunction \Alg$ exhibits $\Alg$ as the localisation of $\Psh
\bbLambda$ at the set $\sfSS$ of spine inclusions \cref{def:spine-lambda}.

In \cref{corollary:colored-algebraic-localization:equivalence}, we show an
equivalence $\sfSS_{n+1,n+2}^{-1} \Psh{\bbOO_{n-k,n+2}} \simeq \Alg$, from
which \cref{th:colored-algebraic-localization} will follow directly.

\begin{lemma}
    \label{lemma:doth-local-equivalences}
    The functor $\dotH_! : \Psh{\bbOO_{n-k, n+2}} \longrightarrow \PshLambda$ takes
    $\sfSS_{n+1} \subseteq \Psh{\bbOO_{n-k, n+2}}^\rightarrow$ to $\sfSS
    \subseteq \PshLambda$ and takes morphisms in $\sfSS_{n+2} \subseteq
    \Psh{\bbOO_{n-k, n+2}}^\rightarrow$ to $\sfSS$-local isomorphisms.
\end{lemma}
\begin{proof}
    \begin{enumerate}
    \item We show that $\dotH_! \sfSS_{n+1} = \sfSS$. Take $\omega \in
    \bbOO_{n+1}$. Then
    \[
        \dotH_! S [\omega] = \dotH_! \colim_{\psi \in \catSS_\omega} O
        [\psi] \cong \colim_{\psi \in \catSS_\omega} \dotH_! O [\psi] =
        \colim_{\psi \in \catSS_\omega} A [\dotH\psi] = S [\omega].
    \]

    \item For $\omega \in \bbOO_{n+2}$, the inclusion $S[\tgt\omega] \longrightarrow
    S[\omega]$ is a relative $\sfSS_{n+1}$-cell complex by
    \cref{lemma:opetopes-technical:spines}. Since $\dotH_!$ preserves colimits, and since
    $\dotH_! \sfSS_{n+1} = \sfSS$, we have that $\dotH_!(S[\tgt \omega]
    \longhookrightarrow S[\omega])$ is a relative $\sfSS$-cell complex, and thus an
    $\sfSS$-local isomorphism. In the square below
    \[
        \squarediagram {\dotH_! S[ \tgt \omega ]} {\dotH_!
          S[\omega]} {\dotH_! O[\tgt \omega]} {\dotH_! O[\omega]} {}
        {\dotH_! \sfS_{\tgt \omega}} {\dotH_! \sfS_\omega} {\dotH_! \tgt }
    \]
    we know that $\dotH_! \tgt = O[ \dotH \tgt ]$ is an isomorphism, and that
    $\dotH_! \sfS_{\tgt \omega} \in \sfSS$ by the previous point. We have just
    shown that the top arrow is an $\sfSS$-local isomorphism. By 3-for-2, we
    conclude that $\dotH_! \sfS_\omega$ is too.
    \qedhere
    \end{enumerate}
\end{proof}

We will prove two crucial lemmas that are key to the proof of
\cref{th:colored-algebraic-localization}. Before doing
so, let us make some preliminary remarks that will make the task easier.

\begin{lemma}
   \label{lemma:prelim-remarks}
    Let $\omega \in \bbOO_{n-k, n+2}$, and let $X \in \sfSS_{n+1, n+2}^{\perp}$.
    Then the following are all spans of isomorphisms.
    \begin{enumerate}
        \item $\bbLambda(\dotH\omega, \dotH\psi) \times X_{\psi}
        \xleftarrow{\id \times X_{\tgt \tgt}} \bbLambda(\dotH\omega, \dotH\psi)
        \times X_{\itree{\psi}} \xrightarrow{\bbLambda(\dotH \omega , \dotH \tgt \tgt)
        \times \id} \bbLambda(\dotH\omega, \dotH\itree{\psi}) \times
        X_{\itree{\psi}}$, where $\psi \in \bbOO_{n-1}$.
        \item $\bbLambda (\dotH\omega, \dotH\psi) \times X_\psi \xleftarrow{\id
        \times X_{\src_{[]}}} \bbLambda (\dotH\omega, \dotH\psi) \times
        X_{\ytree\psi} \xrightarrow{\bbLambda(\dotH\omega , \dotH \src_{[]}) \times
        \id} \bbLambda (\dotH\omega, \dotH\ytree{\psi}) \times X_{\ytree\psi}$,
        where $\psi \in \bbOO_n$.
        \item $ \bbLambda (\dotH\omega, \dotH\tgt\psi) \times X_{\tgt\psi}
        \xleftarrow{\id \times X_{\tgt}} \bbLambda (\dotH\omega, \dotH\tgt\psi)
        \times X_\psi \xrightarrow{\bbLambda (\dotH\omega, \dotH \tgt) \times \id}
        \bbLambda (\dotH\omega, \dotH\psi) \times X_\psi$, where $\psi \in
        \bbOO_{n+2}$
    \end{enumerate}
\end{lemma}
\begin{proof}
    \begin{enumerate}
        \item The first map is an isomorphism by \cref{coroll:y-i:spine-local}
        point (1) and the second map is one by definition of $\dotH$.
        \item The first map is an isomorphism by \cref{coroll:y-i:spine-local}
        point (2) and the second is one by definition of $\dotH$.
        \item The first map is an isomorphism by \cref{lemma:comparison:tgt-spine-local}
        point (1) and the second is one by definition of $\dotH$.
        \qedhere
    \end{enumerate}
\end{proof}

\begin{lemma}
    \label{lemma:prelim-remarks-bis}
    Let $\omega \in \bbOO_{n-k,n+2}$. If $\psi\in\bbOO_{n-k, n-2}$, then
    $\bbLambda(\dotH\omega,\dotH\psi) \cong \bbOO_{n-k, n+2}(\omega, \psi) $.
\end{lemma}
\begin{proof}
   Easy verification.
\end{proof}

The first of the two crucial propositions will provide us one half of an
equivalence between the category $\sfSS_{n+1,n+2} \inv \Psh{\bbOO_{n-k,n+2}}$
and $\Alg$.

\begin{proposition}
    \label{prop:colored-algebraic-localization:crucial-1}
    Let $X \in \Psh{\bbOO_{n-k,n+2}}$. If $\sfSS_{n+1, n+2} \perp X$, then the unit $X \longrightarrow \dotH^*\dotH_!X$ is
    an isomorphism.
\end{proposition}
\begin{proof}
    It suffices to show that for each $\omega \in \bbOO_{n-k, n+2}$, the map
    \[
        X_\omega \longrightarrow
        h^* h_! X_\omega =
        \int^{\psi \in \bbOO_{n-k, n+2}} \bbLambda (\dotH\omega,
        \dotH\psi) \times X_{\psi}
    \]
    is a bijection. We proceed to construct the required inverse via a cowedge
    $\bbLambda (\dotH \omega, \dotH -) \times X_{(-)} \longrightarrow X_\omega$, using a case
    analysis on $\omega \in \bbOO_{n-k,n+2}$.
    \begin{enumerate}
        \item Assume $\omega \in \bbOO_{n-k, n-1}$. We have $\bbLambda
        (\dotH\omega, h-) \cong \bbOO_{n-k, n+2}(\omega, -)$ and this is just
        the change-of-variable formula.
        \item Assume $\omega \in \bbOO_n$. By \cref{lemma:prelim-remarks}
        and \cref{lemma:prelim-remarks-bis}, it suffices to consider the case
        $\psi \in \bbOO_{n+1}$. We have the sequence of morphisms
        \begin{small} \begin{align*}
            &\qquad \bbLambda (\dotH\omega,\dotH\psi) \times X_\psi &&\cong
                \left(
                    \sum_{\substack{\nu \in \bbOO_{n+1} \\ \tgt \nu = \omega}}
                    \Psh{\bbOO_{n-k, n+2}}(S [\nu], S[\psi]) \right) \times
                \Psh{\bbOO_{n-k, n+2}}(S[\psi], X) \\
            &\longrightarrow
                \sum_{\substack{\nu \in \bbOO_{n+1} \\ \tgt \nu = \omega}}
                \Psh{\bbOO_{n-k, n+2}}(S[\nu], X) &&\cong
                \sum_{\substack{\nu \in \bbOO_{n+1} \\ \tgt \nu = \omega}}
                X_\nu \\
            &\stackrel{\tgt}{\longrightarrow} X_\omega .
        \end{align*} \end{small}
        It is straightforward to verify that this defines a cowedge whose
        induced map is the required inverse.

        \item Assume $\omega \in \bbOO_{n+1}$. If $\omega$ is degenerate, say
        $\omega = \itree{\phi}$ for some $\phi \in \bbOO_{n-1}$, then
        $\bbLambda (\dotH \omega, \dotH -) \cong \bbLambda(\dotH \phi, \dotH
        -)$ and we are in a case we have treated before. So let $\omega$ be
        non-degenerate. By \cref{lemma:prelim-remarks,lem:prelim-remarks-bis},
        we may suppose $\psi \in \bbOO_{n, n+1}$. Recall that for every $f \in
        \bbLambda(\dotH\omega, \dotH\psi) $, the diagrammatic
        \cref{lemma:diagramatic-lemma} computes a $\xi \in \bbOO_{n+2}$ and
        $[p] \in \xi^\nodesymbol$ such that $\src_{[p]} \xi = \omega$, $\tgt
        \xi = \psi$ and $\dotH \src_{[p]} = f$. By
        \cref{lemma:comparison:tgt-spine-local}, $X_\xi \cong X_\psi$, and thus
        this gives a function
        \begin{align*}
            \bbLambda(\dotH \omega, \dotH \psi) \times X_\psi
            &\longrightarrow X_\omega \\
            (f, x)
            &\longmapsto \src_{[p]} (x) .
        \end{align*}
        It is straightforward to verify that this assignment defines a cowedge,
        whose associated map is the required inverse.

        \item Assume $\omega \in \bbOO_{n+2}$. Then by definition of $\dotH$,
        $\bbLambda(\dotH \omega, \dotH -) \cong \bbLambda(\dotH \tgt \omega,
        \dotH -)$, and this is the case we have just treated.
        \qedhere
    \end{enumerate}
\end{proof}

\begin{corollary}
    \label{prop:colored-algebraic-localization:crucial-2}
    If $\sfSS_{n+1, n+2} \perp X$, then $\sfSS \perp \dotH_! X$.
\end{corollary}
\begin{proof}
    Recall from \cref{lemma:doth-local-equivalences} that $\sfSS = \dotH_! \sfSS_{n+1}$.
    Let $\omega \in \bbOO_{n+1}$. We have
    \begin{align*}
        \PshLambda (\dotH_! S [\omega], \dotH_! X)
        &\cong \Psh{\bbOO_{n-k, n+2}} (S [\omega], \dotH^* \dotH_! X)
            & \text{since } \dotH_! \dashv \dotH^* \\
        &\cong \Psh{\bbOO_{n-k, n+2}} (S [\omega], X)
            & \text{by \cref{prop:colored-algebraic-localization:crucial-1}} \\
        &\cong \Psh{\bbOO_{n-k, n+2}} (\omega, X)
            & \text{since } \sfSS_{n+1} \perp X \\
        &\cong \Psh{\bbOO_{n-k, n+2}} (\omega, \dotH^* \dotH_! X)
            & \text{by \cref{prop:colored-algebraic-localization:crucial-1}} \\
        &\cong \PshLambda (\dotH_! \omega, \dotH_! X)
            & \text{since } \dotH_! \dashv \dotH^* .
    \end{align*}
\end{proof}

This second crucial proposition will provide the other half of the equivalence
between $\Alg$ and the localisation $\sfSS_{n+1,n+2} \inv
\Psh{\bbOO_{n-k,n+2}}$.

\begin{proposition}
    \label{prop:colored-algebraic-localization:crucial-3}
    Let $Y \in \PshLambda$. If $\sfSS \perp Y$, then the counit
    $\dotH_!\dotH^*Y \longrightarrow Y$ is an isomorphism.
\end{proposition}
\begin{proof}
    It suffices to prove that for each $\lambda \in \bbLambda$, the map
    \begin{equation}
        \label{eq:coend-1}
        (\dotH_! Y)_{\lambda}
        = \int^{\psi \in \bbOO_{n-k, n+2}}
            \bbLambda (\lambda, \dotH\psi) \times Y_{\dotH\psi}
        \longrightarrow Y_\lambda
    \end{equation}
    is a bijection. Since the map\footnote{We use the notation $a \otimes b$ to
    refer to elements in a coend of sets of the form $\int^{c\in\catCC} A_c
    \times B_c$, with $A \in [\catCC^\op,\Set]$ and $B \in [\catCC,\Set]$, our
    motivation being the tensor product of modules.}
    \begin{align*}
      Y_\lambda &\longrightarrow \int^{\psi}
        \bbLambda (\lambda, \dotH\psi) \times Y_{\dotH \psi} \\
      y &\longmapsto \id \otimes y
    \end{align*}
    is clearly a section, it is enough to prove that it is is surjective. The
    map is well defined, as $\dotH$ is surjective on objects and it is easy to
    verify that it is independent of the choice of an antecedent $\dotH\nu =
    \lambda$. We need to show that for every $\psi \in \bbOO_{n-k, n+2}$, every
    $f \otimes y \in \bbLambda (\lambda, \dotH\psi) \times Y_{\dotH\psi} $ is
    equal, in the colimit \eqref{eq:coend-1}, to $\id \otimes y' \in \bbLambda
    (\lambda,\lambda) \times Y_{\lambda}$ for some $y' \in Y_{\lambda}$.
    \begin{enumerate}
        \item Assume $\lambda = \optPolyFun\phi$, with $\phi \in
        \bbOO_{n-k, n-1}$. Then $\bbLambda(\lambda, \dotH -) \cong
        \bbOO_{n-k, n+2} (\phi,-)$ and any pair $f \otimes y \in \bbOO_{n-k, n+2}
        (\phi,\psi) \times Y_\psi$ is related via a zig-zag relation
        \[
            \begin{tikzcd}
                f \otimes y
                    \ar[r, leftrightsquigarrow, "f" above] &
                \id_\phi \otimes (\dotH f) (y)
            \end{tikzcd}
        \]
        to an element of the required form.

        \item Assume $\lambda = \optPolyFun\omega$, with $\omega \in
        \bbOO_{n}$. By \cref{lemma:prelim-remarks} and
        \cref{lemma:prelim-remarks-bis}, we may consider only the case where
        $\psi\in \bbOO_{n+1}$. Note that $\bbLambda(\lambda, \dotH \psi) \cong
        \bbLambda(\dotH \ytree{\omega}, \dotH \psi)$. Then $f \otimes y \in
        \bbLambda(\dotH \ytree{\omega}, \dotH -)$ is related via a zig-zag
        relation
        \[
            \begin{tikzcd}
                f \otimes y
                    \ar[r, leftrightsquigarrow, "\src_{[p]}" above] &
                \id_{\dotH \ytree\omega} \otimes (\dotH \src_{[p]}) (y)
            \end{tikzcd}
        \]
        where we use \cref{lemma:diagramatic-lemma} to obtain $\src_{[p]} :
        \ytree \omega \longrightarrow \xi$ such that $\dotH \src_{[p]} = f$.
        \item The case $\lambda = \optPolyFun S[\omega ]$, with $\omega \in
        \bbOO_{n+1}$ identical to the previous one.
        \qedhere
    \end{enumerate}
\end{proof}

Let the localisation of $\Psh{\bbOO_{n-k, n+2}}$ at the set of spine inclusions
$\sfSS_{n+1, n+2}$ be denoted
\[
    v : \Psh{\bbOO_{n-k, n+2}}
    \adjunction
    \sfSS_{n+1, n+2} \inv \Psh{\bbOO_{n-k, n+2}} : N_v
\]
On the other hand, recall from \cref{prop:alg-shape-nerve} that we have a
localisation $u : \PshLambda \adjunction \Alg : N_u$. We are now
well-equipped to prove that $\Alg$ is equivalent to the localised category
$\sfSS_{n+1, n+2} \inv \Psh{\bbOO_{n-k, n+2}}$.

\begin{proposition}
    \label{prop:doth-restricts}
    The adjunction $\dotH_! : \Psh{\bbOO_{n-k, n+2}} \adjunction \PshLambda :
    \dotH^*$ restricts to an adjunction $\tildH_! \dashv \tildH^*$, as shown
    below.
    \[
        \begin{tikzcd}
            \sfSS_{n+1, n+2} \inv \Psh{\bbOO_{n-k, n+2}}
                \ar[d, hook, "v"']
                \ar[r, shift left = .4em, "\tildH_!"] &
            \Alg
                \ar[d, hook, "N_u"]
                \ar[l, shift left = .4em, "\tildH^*", "\perp" above] \\
            \Psh{\bbOO_{n-k, n+2}}
                \ar[r, shift left = .4em, "\dotH_!"] &
            \PshLambda .
                \ar[l, shift left = .4em, "\dotH^*", "\perp" above]
        \end{tikzcd}
    \]
\end{proposition}
\begin{proof}
    For all $Y \in \Alg \simeq \sfSS\inv \PshLambda$, by
    \cref{lemma:doth-local-equivalences}, we have that $h_!\sfSS_{n+1,n+2}
    \perp N_uY$, or equivalently, $\sfSS_{n+1,n+2} \perp h^*N_u Y$. Thus
    $h^*N_u$ factors through $\sfSS_{n+1, n+2} \inv \Psh{\bbOO_{n-k, n+2}}$.
    Next, by \cref{prop:colored-algebraic-localization:crucial-2}, $\dotH_!
    N_v$ factors through $\Alg$.
\end{proof}

\begin{corollary}
    \label{corollary:colored-algebraic-localization:equivalence}
    The adjunction $\tildH_!:\sfSS_{n+1, n+2}\inv\Psh{\bbOO_{n-k, n+2}}
    \adjunction \Alg:\tildH^*$ is an equivalence.
\end{corollary}
\begin{proof}
    This is a direct consequence of
    \cref{prop:colored-algebraic-localization:crucial-1,prop:colored-algebraic-localization:crucial-3}.
\end{proof}

\begin{notation}
    \label{not:algebraic-realization-adjunction}
    With \cref{corollary:colored-algebraic-localization:equivalence} in hand,
    let
    \[
        h : \PshO \adjunction \Alg : N .
    \]\index{$h$}\index{$N$}
    be the composite adjunction
    \[
        \PshO
        \adjunction \Psh{\bbOO_{n-k, n+2}}
        \stackrel{v}{\adjunction} \sfSS_{n+1, n+2}^{-1} \Psh{\bbOO_{n-k,n+2}}
        \stackrel{\tildH_!}{\adjunction} \Alg .
    \]
\end{notation}

\begin{theorem}
    [Nerve theorem for opetopic algebras]
    \label{th:colored-algebraic-localization}
    The reflective adjunction $h : \PshO \longrightarrow \Alg : N$ exhibits $\Alg$ as the
    Gabriel--Ulmer localisation (\cref{sec:prelim}) of $\PshO$ at the set of arrows $\sfAA$. That
    is $\Alg \simeq \sfAA^{-1} \PshO$.
\end{theorem}
\begin{proof}
    By a general fact about composite localisations, the reflective adjunction
    $h : \PshO \longrightarrow \Alg : N$ exhibits $\Alg$ as the localisation of $\PshO$ at
    the set $\sfOO_{<n-k}\cup \sfSS_{n+1, n+2} \cup \sfBB_{>n+2}$. The result
    then follows from \cref{lemma:opetopes-technical:h-lift}.
\end{proof}

\begin{corollary}
    \label{coroll:h-surjective}
    We have $\bbLambda = h \bbOO_{n-k,n+2}$, i.e. $h = \dotH : \bbOO_{n-k,n+2}
    \longrightarrow \bbLambda$ is surjective on objects (by definition) and on
    morphisms.
\end{corollary}
\begin{proof}
    Let $\omega, \omega' \in \bbOO_{n-k, n+2}$.
    \begin{enumerate}
        \item If $\dim \omega, \dim \omega' < n-1$, then $h \omega = \omega$
        and $h \omega' = \omega'$ as presheaves over $\bbOO_{n-k, n+2}$,
        and thus
        \[
            \bbLambda (h \omega, h \omega')
            = h \Psh{\bbOO_{n-k, n+2}} (\omega, \omega')
            = h \bbOO (\omega, \omega') .
        \]

        \item If $\dim \omega < n-1$ and $\dim \omega' \geq n-1$, then since
        $h \omega = \omega$ we have
        \[
            \bbLambda (h \omega, h \omega')
            \cong \Psh{\bbOO_{n-k, n+2}} (\omega, h \omega')
            \cong h \Psh{\bbOO_{n-k, n+2}} (\omega, \omega')
            \cong h \bbOO (\omega, \omega') ,
        \]
        where the second isomorphism comes from the fact that for $X \in
        \Psh{\bbOO_{n-k, n}}$, $\optPolyFun^n X_{< n} = X_{<n}$.

        \item If $\dim \omega \geq n-1$ and $\dim \omega' < n-1$, then
        $\bbLambda (h \omega, h \omega') = \emptyset$.

        \item Lastly, assume $\dim \omega, \dim \omega' \geq n-1$. By
        \cref{coroll:y-i:spine-local}, if $\dim \omega = n-1$, then $h$ maps $\tgt
        \tgt : \omega \longrightarrow \itree{\omega}$ to an isomorphism, and if
        $\dim \omega = n$, then $h$ maps $\src_{[]} : \omega \longrightarrow
        \ytree{\omega}$ to an isomorphism. By
        \cref{lemma:comparison:tgt-spine-local}, if $\dim \omega = n+2$, then $h$
        maps $\tgt : \tgt \omega \longrightarrow \omega$ to an isomorphism as
        well. Hence, without loss of generality, we may assume that $\dim
        \omega = \dim \omega' = n+1$. If $\omega$ is non degenerate, then every
        morphism in $\bbLambda (h \omega, h \omega')$ is diagrammatic, thus in
        the image of $h$. Otherwise, if $\omega = \itree{\phi}$ for some $\phi
        \in \bbOO_{n-1}$, then
        \[
            \bbLambda (h \omega, h \omega')
            \cong \bbLambda (h \phi, h \omega')
            \cong h \bbOO (\phi, \omega') .
        \]
        where the first isomorphism comes from \cref{coroll:y-i:spine-local}.
        \qedhere
    \end{enumerate}
\end{proof}


\section{Algebraic trompe-l'\oe{}il}
\label{sec:trome-loeil}

As we saw in \cref{sec:algebraic-realization}, for each $k, n \geq 1$, we have
a notion of $k$-coloured opetopic $n$-algebra. For such an algebra $B \in \Oalg
k n$, operations are $n$-cells (so that their shapes are $n$-opetopes), and
colours are cells of dimension $n-k$ to $n-1$, thus the ``colour space'' is
stratified over $k$ dimensions. Notable examples include
\[
    \Cat \simeq \Oalg 1 1 ,
    \qqquad
    \Op_{\mathrm{col}} \simeq \Oalg 1 2 .
\]
But are all $\Oalg k n$ fundamentally different?

In this section, we answer this question negatively: in a sense that we make
precise, the most ``algebraically rich'' notion of opetopic algebra is given in
the case $(k, n) = (1, 3)$. Although opetopes can be arbitrarily complex, the
algebraic data can be expressed by much simpler $3$-opetopes, a.k.a. trees. We
call this phenomenon \emph{algebraic trompe-l'\oe{}il}\index{algebraic
trompe-l'\oe{}il}, a french expression that literally translates as
``fools-the-eye''. And indeed, the eye is fooled in two ways: by colour
(\cref{prop:algebraic-trompe-loeil:colour}) and shape
(\cref{prop:algebraic-trompe-loeil:shape}). In the former, we argue that the
colours of an algebra $B \in \Oalg k n$, expressing how operations may or may
not be composed, only need $1$ dimension, and thus that cells of dimension less
than $n-1$ do not bring new algebraic data, only geometrical one. For the
latter, recall from \cref{sec:opetopes:definition} that opetopes are trees of
opetopes. In particular, $3$-opetopes are just plain trees, and $\bbOO_3$
already contains all the possible underlying trees of all opetopes.
Consequently, operations of $B$, which are its $n$-cells, may be considered as
$3$-cells in a very similar $3$-algebra $B_\dagger$. Finally, we combine those
two results in \cref{th:algebraic-trompe-loeil}, which states that an algebra
$B \in \Oalg k n$ is exactly a presheaf $B \in \Psh{\bbOO_{n-k, n}}$ with a
$1$-coloured $3$-algebra structure on $B_{n-1, n,\dagger}$.

\subsection{Colour}

For $B \in \Oalg k n$, recall that the colours of $B$ are its cells of
dimension $n-k$ to $n-1$. They express which operations ($n$-cells) of $B$ may
or may not be composed. However, since that criterion only depends on
$(n-1)$-cells, constraints expressed by lower dimensional cells should be
redundant. In \cref{prop:algebraic-trompe-loeil:colour}, we show that this is
indeed the case, in that the algebra structure on $B$ is completely determined
by a $1$-coloured $n$-algebra structure on $B_{n-1, n}$.

\begin{lemma}
    \label{lemma:spine-detruncation}
    Let $k, n \geq 1$, and $\nu \in \bbOO_{n+1}$. Then
    \[
        S [\nu]_{n-k, n} \cong \iota_! (S [\nu]_{n-1, n}) ,
    \]
    where $\iota_!$ is the left adjoint to the truncation $\Psh{\bbOO_{n-k, n}}
    \longrightarrow \Psh{\bbOO_{n-1, n}}$.
\end{lemma}
\begin{proof}
    It follows from the fact that $S [\nu]$ is completely determined by the
    incidence relation of the $n$- and $(n-1)$-faces of $\nu$.
\end{proof}

\begin{proposition}
    \label{prop:truncation-algebra}
    For $X \in \Psh{\bbOO_{n-k, n}}$ we have $\optPolyFun^n (X_{n-1, n}) \cong
    (\optPolyFun^n X)_{n-1, n}$. Consequently, the truncation functor
    $(-)_{n-1, n} : \Psh{\bbOO_{n-k, n}} \longrightarrow \Psh{\bbOO_{n-1, n}}$
    lifts as
    \begin{equation}
        \diagramsize{2}{5}
        \label{eq:algebraic-trompe-loeil:colour}
        \squarediagram
            {\Oalg k n}{\Oalg 1 n}{\Psh{\bbOO_{n-k, n}}}
                {\Psh{\bbOO_{n-1, n}} .}
            {(-)_{n-1, n}}{}{}{(-)_{n-1, n}}
        \diagramsize{2}{3}
    \end{equation}
\end{proposition}
\begin{proof}
    First, $\optPolyFun^n (X_{n-1, n})_{n-1} = X_{n-1} = (\optPolyFun^n
    X)_{n-1}$. Then, for $\psi \in \bbOO_n$, we have
    \begin{align*}
        \optPolyFun^n (X_{n-1, n})_\psi
        &= \sum_{\substack{\nu \in \bbOO_{n+1} \\ \tgt \nu = \psi}}
            \Psh{\bbOO_{n-1, n}} (S [\nu], X_{n-1, n}) \\
        &\cong \sum_{\substack{\nu \in \bbOO_{n+1} \\ \tgt \nu = \psi}}
            \Psh{\bbOO_{n-1, n}} (\iota_! (S [\nu]), X) \\
        &\cong \sum_{\substack{\nu \in \bbOO_{n+1} \\ \tgt \nu = \psi}}
            \Psh{\bbOO_{n-1, n}} (S [\nu], X)
            & \text{by \cref{lemma:spine-detruncation}} \\
        &= (\optPolyFun^n X)_\psi .
    \end{align*}
\end{proof}

\begin{proposition}
    \label{prop:algebraic-trompe-loeil:colour}
    The square \eqref{eq:algebraic-trompe-loeil:colour} is a pullback. That is, a
    $\optPolyFun^n$-algebra structure on $X \in \Psh{\bbOO_{n-k, n}}$ is
    completely determined by a $\optPolyFun^n$-algebra structure on $X_{n-1,
    n}$.
\end{proposition}
\begin{proof}
    Let $X \in \Psh{\bbOO_{n-k, n}}$. By \cref{prop:truncation-algebra}, a
    $\optPolyFun^n$-algebra structure on $X$ restricts to one on $X_{n-1, n}$.
    Conversely, since $(\optPolyFun^n X)_{<n} = X_{<n}$, a
    $\optPolyFun^n$-algebra structure on $X_{n-1, n}$ extends to one on $X$.
    Since the truncation functor $(-)_{n-1, n} : \Oalg k n
    \longrightarrow \Oalg 1 n$ is faithful, it establishes a
    bijective correspondence between the algebra structures on $X$ and on
    $X_{n-1, n}$.
\end{proof}

\subsection{Shape}

We start by defining a functor $(-)_\dagger : \bbOO_{n-1, n} \longrightarrow
\bbOO_{2, 3}$, for $n \geq 1$, mapping an $n$-opetope $\omega$ to the unique
$3$-opetope $\omega_\dagger$ having the same underlying polynomial tree, i.e.
$\underlyingtree{\omega_\dagger} \cong \underlyingtree{\omega}$ (see notations
of \cref{sec:trees}).

\begin{definition}
    \label{def:dagger}
    \begin{enumerate}
        \item If $n = 1$, then $(-)_\dagger$ simply maps $\bbOO_{0, 1} =
        \left( \optZero
        \xrightarrow{\src_*, \tgt} \optOne \right)$ to the diagram $\left( \optInt{0}
        \xrightarrow{\src_{[]}, \tgt} \ytree{\optInt{0}} \right)$.

        \item Assume now that $n \geq 2$. Recall that a $3$-opetope is a
        $\optPolyFun^1$-tree, where $\optPolyFun^1$ is given by
        \[
            \polynomialfunctor
                {\{ \optOne \}}{E_2}{\bbOO_2}{\{ \optOne \} ,}
                {\src}{p}{\tgt}
        \]
        with $\bbOO_2 = \left\{ \optInt{n} \mid n \in \bbNN \right\}$ and $E_2
        (\optInt{n}) = \optInt{n}^\nodesymbol$. Let $f : \optPolyFun^{n-2}
        \longrightarrow \optPolyFun^1$ be given by
        \[
            \begin{tikzcd}
                \bbOO_{n-2}
                    \ar[d, "f_0" left] &
                E_{n-2}
                    \ar[l, "\src" above]
                    \ar[d, "f_2"]
                    \ar[r, "p"]
                    \pullbackcorner &
                \bbOO_{n-1}
                    \ar[d, "f_1"]
                    \ar[r, "\tgt"] &
                \bbOO_{n-2}
                    \ar[d, "f_0"] \\
                \{ \optOne \} &
                E_2
                    \ar[l, "\src" above]
                    \ar[r, "p"] &
                \bbOO_2
                    \ar[r, "\tgt"] &
                \{ \optOne \} ,
            \end{tikzcd}
        \]
        where $f_1 (\psi) = \optInt{m}$, for $m = \# \psi^\nodesymbol$ the
        number of source faces of $\psi$, and where $f_2$ is fiberwise
        increasing. This morphism of polynomial functors induces a functor $f_*
        : \bbOO_n = \tree \optPolyFun^{n-2} \longrightarrow \tree \optPolyFun^1
        = \bbOO_3$ mapping an $n$-opetope to its underlying tree, seen as a
        $3$-opetope. Explicitely,
        \[
            f_* \itree{\phi} = \itree{\optOne},
            \qqquad
            f_* \left( \ytree{\psi} \biggraft_{[[p_i]]} \nu_i \right) =
                \optInt{m} \biggraft_{[[*^i]]} f_* (\nu_i),
        \]
        where $\phi \in \bbOO_{n-2}$, $\psi \in \bbOO_{n-1}$, $\psi^\nodesymbol
        = \left\{ [p_0] \prec \cdots \prec [p_{m-1}] \right\}$, and $\nu_0,
        \ldots, \nu_{m-1} \in \bbOO_n$. For $\omega \in \bbOO_n$, since
        $\omega$ and $\omega_\dagger$ have the same underlying tree, they have
        the same number of source faces: $\hash \omega^\nodesymbol = \hash
        \omega_\dagger^\nodesymbol$, and we write $a_\omega$ for the unique
        increasing map $\omega^\nodesymbol \longrightarrow
        \omega_\dagger^\nodesymbol$ with respect to the lexicographical order
        $\prec$\footnote{Intuitively, $a_\omega$ maps a node of the underlying
        tree $\underlyingtree{\omega}$ of $\omega$ to that same node in
        $\underlyingtree{\omega_\dagger}$, but using addresses. Since the
        source faces of $\omega$ and $\omega_\dagger$ are not the same,
        $a_\omega$ is not strictly speaking an identity, but rather a
        conversion of a ``walking instruction in the tree $\omega$'' (which is
        what an address is) to one in $\omega_\dagger$. Explicitely, a node
        address $[[q_1] \cdots [q_k]] \in \omega^\nodesymbol$ (with $[q_{i+1}]
        \in \src_{[[q_1] \cdots [q_i]]} \omega$) is mapped to $[f_{2, \src_{[]}
        \omega} [q_1] \:\: \cdots \:\: f_{2, \src_{[[q_1] \cdots [q_{k-1}]]}
        \omega} [q_k]]$.}.

        \item Define now $(-)_\dagger : \bbOO_{n-1, n} \longrightarrow
        \bbOO_{2, 3}$\index{$(-)_\dagger$} as follows: for $\psi \in
        \bbOO_{n-1}$ and $\omega \in \bbOO_n$
        \begin{enumerate}
            \item $\psi_\dagger = f_1 (\psi)$ as above;
            \item $\omega_\dagger = f_* (\omega)$ as above;
            \item we have $(\tgt \omega)_\dagger = \tgt \omega_\dagger$, so let
            $\left( \tgt \omega \xrightarrow{\tgt} \omega \right)_\dagger =
            \left((\tgt \omega)_\dagger \xrightarrow{\tgt} \omega_\dagger
            \right)$;
            \item for $[p] \in \omega^\nodesymbol$, we have $(\src_{[p]}
            \omega)_\dagger = \src_{a_\omega [p]} \omega_\dagger$, so let
            $\left(\src_{[p]} \omega \xrightarrow{\src_{[p]}} \omega
            \right)_\dagger =
            \left((\src_{[p]} \omega)_\dagger \xrightarrow{\src_{a_\omega [p]}}
            \omega_\dagger \right)$.
        \end{enumerate}
    \end{enumerate}
\end{definition}

\begin{example}
    Consider the $4$-opetope $\omega$, represented graphically and in tree form
    below:
    \[
        \tikzinput[.8]{opetope-4-graphical}{1}
        \qqquad
        \tikzinput[.8]{opetope-4-tree}{1,1.annotations}
    \]
    where $\psi_1$ and $\psi_2$ are the $3$-opetopes on the top right and top
    left hand corner respectively. Then $\omega_\dagger$ is as follows:
    \[
        \tikzinput[.8]{opetope-3-graphical}{big-dagger}
        \qqquad
        \tikzinput[.8]{opetope-4-tree}{1,1.dagger}
    \]
    Although the graphical representations of $\omega$ and $\omega_\dagger$
    look nothing alike, note that their underlying undecorated trees are
    identical.
\end{example}

\begin{notation}
    \label{not:dagger}
    We abuse notations and let $(-)_\dagger : \Psh{\bbOO_{n-1, n}}
    \longrightarrow \Psh{\bbOO_{2, 3}}$ be the left Kan extension of
    $\bbOO_{n-1, n} \xrightarrow{(-)_\dagger} \bbOO_{2, 3} \longrightarrow
    \Psh{\bbOO_{2, 3}}$ along the Yoneda embedding. Explicitely, for $X \in
    \Psh{\bbOO_{n-1, n}}$, we have
    \[
        X_{\dagger, \optInt{m}} =
            \sum_{\substack{
                \psi \in \bbOO_{n-1} \\
                \psi_\dagger = \optInt{m}
            }} X_\psi ,
        \qqquad
        X_{\dagger, \gamma} =
            \sum_{\substack{
                \omega \in \bbOO_{n-1} \\
                \omega_\dagger = \gamma
            }} X_\omega ,
    \]
    with $m \in \bbNN$ and $\gamma \in \bbOO_3$.
\end{notation}

\begin{remark}
    Clearly, $(-)_\dagger$ is faithful, and if $n \leq 3$, then $(-)_\dagger$
    is also injective on object. Note that this is no longer the case if $n
    \geq 4$, as distinct $n$-opetopes may have the same underlying tree.
\end{remark}

\begin{notation}
    Let $\catCC$ be a small category, and $X \in \Psh{\catCC}$ be presheaf over
    $\catCC$. There is a canonical projection $\shape{(-)} : \catCC/X
    \longrightarrow \catCC$\index{$\shape{(-)}$|see {shape}}, mapping $x \in
    X_c$ to its \emph{shape}\index{shape} $c \in \catCC$. We may then see the
    category of elements $\catCC / X$ of $X$ as having objects elements of
    $\sum_{c \in \catCC} X_c$, and a morphism $f : x \longrightarrow y$ is a
    morphism $f : \shape{x} \longrightarrow \shape{y}$ in $\catCC$ such that
    $f(y) = x$. A morphism $g : X \longrightarrow Y$ of presheaves over
    $\catCC$ then amounts to a functor $g : \catCC/ X \longrightarrow \catCC/
    Y$ that preserves shapes.
\end{notation}

\begin{remark}
    Take $n \geq 1$. For $X \in \Psh{\bbOO_{n-1, n}}$, note that there is a
    canonical isomorphism $\Psh{\bbOO_{n-1, n}} / X \longrightarrow
    \Psh{\bbOO_{2, 3}} / X_\dagger$, which is the identity on objects, maps
    $\src_{[p]} : x \longrightarrow y$ to $\src_{a_\omega [p]} x
    \longrightarrow y$, where $\omega = \shape{y}$, and target embeddings to
    target embeddings.
\end{remark}

\begin{lemma}
    \label{lemma:technical-dagger}
    \begin{enumerate}
        \item For $\nu \in \bbOO_{n+1}$, there exists a unique $4$-opetope
        $\nu' \in \bbOO_4$ such that $S [\nu]_{n-1, n,\dagger} \cong S
        [\nu']$.
        \item Let $X \in \Psh{\bbOO_{n-1, n}}$, $\nu \in \bbOO_4$, and $f :
        S [\nu] \longrightarrow X_\dagger$. Then there exists a unique
        $\nu' \in \bbOO_{n+1}$ and $f' : S [\nu'] \longrightarrow X$ such
        that $S [\nu']_{n-1, n,\dagger} = S [\Lambda]$, and
        $f'_\dagger = f$.
    \end{enumerate}
\end{lemma}
\begin{proof}
    \begin{enumerate}
        \item If $\nu = \itree{\phi}$ for $\phi \in \bbOO_{n-1}$, let $\nu' =
        \itree{\phi_\dagger}$. If $\nu = \ytree{\omega} \biggraft_{[[p_i]]}
        \nu_i$, let
        \[
            \nu' = \ytree{\omega_\dagger} \biggraft_{[a_\omega [p_i]]} \nu_i' ,
        \]
        where the $\nu_i'$ are given by induction. The graftings are well
        defined since
        \[
            \tgt \src_{[]} \nu_i'
            = \tgt (\src_{[]} \nu_i)_\dagger
            = (\tgt \src_{[]} \nu_i)_\dagger
            = (\src_{[p_i]} \omega)_\dagger
            = \src_{a_\omega [p_i]} \omega_\dagger .
        \]
        The isomorphism $S [\nu]_{n-1, n,\dagger} \cong
        S [\nu']$ can easily be shown by induction on the
        structure of $\nu$ and using
        \cref{lemma:opetopes-technical:spine-pushout}.

        \item For $\nu^\nodesymbol = \{ [p_1], \ldots, [p_m] \}$, $f$ maps
        $[p_i]$ to a cell $x_i \in X_{\dagger, 2} = X_{n-1}$, and let $\psi_i
        \in \bbOO_{n-1}$ be the shape of $x_i$ as a cell of $X$. If $[p_i] =
        [p_j[q]]$ for some $j$ and $[q]$, then $\src_{[q]} x_j = \tgt x_i$ in
        $X_\dagger$, so $\src_{a_{\psi_j}^{-1} [q]} x_j = \tgt x_i$ in $X$, and
        in particular, $\src_{a_{\psi_j}^{-1} [q]} \psi_j = \tgt \psi_i$.
        Consequently, the $\psi_i$s may be grafted together into a
        $(n+1)$-opetope $\nu'$ such that $\nu'_\dagger = \nu$, and
        $\src_{a_{\nu'}^{-1} [p_i]} = \psi_i$. Define $f' : S [\nu']
        \longrightarrow X$ mapping $\src_{a_{\nu'}^{-1} [p_i]} \nu'$ to $x_i$,
        and observe that $f'_\dagger = f$.
        \qedhere
    \end{enumerate}
\end{proof}

\begin{proposition}
    \label{prop:dagger-algebra}
    For $X \in \Psh{\bbOO_{n-1, n}}$ we have $\optPolyFun^3 (X_\dagger) \cong
    (\optPolyFun^n X)_\dagger$. Consequently, the functor $(-)_\dagger$ lifts as
    \begin{equation}
        \label{eq:algebraic-trompe-loeil:shape}
        \squarediagram
            {\Oalg 1 n}{\Oalg 1 3}{\Psh{\bbOO_{n-1, n}}}{\Psh{\bbOO_{2, 3}} .}
            {(-)_\dagger}{}{}{(-)_\dagger}
    \end{equation}
\end{proposition}
\begin{proof}
    First, $\optPolyFun^3 (X_\dagger)_2 = X_{\dagger, 2} = X_{n-1} =
    (\optPolyFun^n X)_{n-1} = (\optPolyFun^n X)_{\dagger, 2}$. Then,
    \begin{align*}
        \optPolyFun^3 (X_\dagger)_3
        &= \sum_{\nu \in \bbOO_4} \Psh{\bbOO_{2, 3}} (S [\nu], X_\dagger) \\
        &\cong \sum_{\nu \in \bbOO_{n+1}} \Psh{\bbOO_{n-1, n}} (S [\nu], X) & \text{by \cref{lemma:technical-dagger}} \\
        &= (\optPolyFun^n X)_n = (\optPolyFun^n X)_{\dagger, 3} .
    \end{align*}
\end{proof}

\begin{lemma}
    \label{lemma:dagger-algebra-structure}
    Let $X \in \Psh{\bbOO_{n-1, n}}$ and $m : \optPolyFun^n X \longrightarrow X$.
    Then $m$ is an algebra structure on $X$ if and only if $m_\dagger :
    \optPolyFun^3 X_\dagger \longrightarrow X_\dagger$ is an algebra structure
    on $X_\dagger$.
\end{lemma}
\begin{proof}
    Clearly, $(-)_\dagger$ maps the multiplication $\mu^n : \optPolyFun^n
    \optPolyFun^n \longrightarrow \optPolyFun^n$ to $\mu^3$, and the unit
    $\eta^n : \id \longrightarrow \optPolyFun^n$ to $\eta^3$. Since
    $(-)_\dagger$ is faithful, the square on the left commutes if and only if
    the square on the right commutes
    \[
        \diagramsize{2}{4}
        \squarediagram
            {\optPolyFun^n \optPolyFun^n X}{\optPolyFun^n X}{\optPolyFun^n X}
                {X ,}
            {\optPolyFun^n m}{\mu^n}{m}{m}
        \qqquad
        \squarediagram
            {\optPolyFun^3 \optPolyFun^3 X_\dagger}{\optPolyFun^3 X_\dagger}
                {\optPolyFun^3 X_\dagger}{X_\dagger ,}
            {\optPolyFun^3 m_\dagger}{\mu^3}{m_\dagger}{m_\dagger}
        \diagramsize{2}{3}
    \]
    and likewise for the diagram involving $\eta^n$ and $\eta^3$.
\end{proof}

\begin{proposition}
    \label{prop:algebraic-trompe-loeil:shape}
    The square \eqref{eq:algebraic-trompe-loeil:shape} is a pullback. That is, a
    $\optPolyFun^n$-algebra structure on $X \in \Psh{\bbOO_{n-1, n}}$ is
    completely determined by a $\optPolyFun^3$-algebra structure on $X_\dagger$.
\end{proposition}
\begin{proof}
    Let $X \in \Psh{\bbOO_{n-1, n}}$. By \cref{prop:dagger-algebra}, a
    $\optPolyFun^n$-algebra structure on $X$ induces a $\optPolyFun^3$-algebra
    structure on $X_\dagger$.

    Conversely, let $m : \optPolyFun^3 X_\dagger \longrightarrow X_\dagger$ be
    a $\optPolyFun^3$-algebra structure on $X_\dagger$, and define $m' :
    \optPolyFun^n X \longrightarrow X$ as the identity in dimension $n-1$, and
    mapping $f : S [\nu] \longrightarrow X$ to $m(f_\dagger) \in X_{\dagger, 2}
    = X_{n-1}$. Recall that $f_\dagger$ is a map of the form $S [\nu']
    \longrightarrow X_\dagger$, for some $\nu'$ such that $\tgt \nu' = (\tgt
    \nu)_\dagger$, and thus $m'$ is a map of opetopic sets. By
    \cref{lemma:dagger-algebra-structure}, it is a $\optPolyFun^n$-algebra
    structure on $X$.

    Since $(-)_\dagger$ is faithful, it establishes a bijective correspondence
    between the $\optPolyFun^n$-algebra structures on $X$ and the
    $\optPolyFun^3$-algebra structures on $X_\dagger$.
\end{proof}

\begin{theorem}
    [Algebraic trompe-l'\oe{}il]
    \label{th:algebraic-trompe-loeil}
    The following square is a pullback\index{algebraic trompe-l'\oe{}il}:
    \begin{equation}
        \diagramsize{2}{5}
        \squarediagram
            {\Oalg k n}{\Oalg 1 3}{\Psh{\bbOO_{n-k, n}}}{\Psh{\bbOO_{2, 3}} .}
            {(-)_{n-1, n,\dagger}}{}{}{(-)_{n-1, n,\dagger}}
            \diagramsize{2}{3}
    \end{equation}
    In otherwords, a $\optPolyFun^n$-algebra structure on $X \in
    \Psh{\bbOO_{n-k, n}}$ is completely determined by a $\optPolyFun^3$-algebra
    structure on $(X_{n-1, n})_\dagger$.
\end{theorem}
\begin{proof}
    This is a direct consequence of \cref{prop:algebraic-trompe-loeil:colour},
    \cref{prop:algebraic-trompe-loeil:shape}, and the pasting lemma for
    pullbacks.
\end{proof}

\appendix

\section{Omitted proofs}
\label{sec:omitted-proofs}

\begin{proof}
    [Proof of \cref{lemma:colored-zn:pra-monad-structure}]
    \begin{enumerate}
        \item Let $P$ be the pullback of the cospan $1 \xrightarrow{\eta_1}
        \optPolyFun^n 1 \xleftarrow{\optPolyFun^n !} \optPolyFun^n X$. Since
        $(\optPolyFun_n X)_{<n} = X_{<n}$, we trivially have $P_{<n} = X_{<n}$.
        Next, for $\omega \in \bbOO_n$, we have
        \[
            P_\omega
            = \left\{ x \in \optPolyFun^n X \mid \optPolyFun^n ! (x) = \ytree{\omega} \right\} \\
            = \Psh{\bbOO_{n-k, n}} (S [\ytree{\omega}], X) \\
            = X_\omega
        \]
        as $S [\ytree{\omega}] = O [\omega]$.

        \item Let $P$ be the pullback of the cospan $\optPolyFun^n
        \optPolyFun^n 1 \xrightarrow{\mu_1} \optPolyFun^n 1
        \xleftarrow{\optPolyFun^n !} \optPolyFun^n X$. As before, since
        $(\optPolyFun_n X)_{<n} = X_{<n}$, we trivially have $P_{<n} = X_{<n} =
        (\optPolyFun^n \optPolyFun^n X)_{<n}$. Recall from \cref{lemma:zn-zn}
        that as a polynomial functor, $\optPolyFun^n \optPolyFun^n : \Set /
        \bbOO_n \longrightarrow \Set / \bbOO_n$ is given by
        \[
            \polynomialfunctor
                {\bbOO_n}
                {E}
                {\bbOO_{n+2}^{(2)}}
                {\bbOO_n ,}
                {\edg}
                {}
                {\tgt \tgt}
        \]
        where $\bbOO_{n+2}^{(2)}$ is the set of $(n+2)$-opetopes of height at
        most $2$. Then, for $\omega \in \bbOO_n$, we have:
        \[
            P_\omega
            = \left\{ (\xi, x) \mid x : \Lambda^{\tgt} [\nu] \rightarrow X, \xi \in \bbOO_{n+2}^{(2)}, \tgt \xi = \nu \right\} \\
            = \left\{ x : S [\tgt \xi] \mid \xi \in \bbOO_{n+2}^{(2)} \right\} \\
            = \optPolyFun^n \optPolyFun^n X_\omega .
        \]
    \end{enumerate}
\end{proof}

\begin{proof}
    [Proof of \cref{lemma:colored-zn:pra-monad-structure:1}]
    For $X \in \Psh{\bbOO_{n-k, n}}$, $(\optPolyFun^n X)_{<n} = X_{<n}$, thus
    all diagrams commute trivially in dimension $<n$.
    \begin{enumerate}
        \item Let $\omega \in \bbOO_n$ and $\nu \in \optPolyFun^n 1_\omega$,
        i.e. $\nu \in \bbOO_{n+1}$ such that $\tgt \nu = \omega$. Then
        \[
            \mu_1 \eta_{\optPolyFun^n 1} (\nu)
            = \mu_1 \left( \ytree{\ytree{\tgt \nu}} \graft_{[[]]} \ytree{\nu} \right)
            = \tgt \left( \ytree{\ytree{\tgt \nu}} \graft_{[[]]} \ytree{\nu} \right)
            = \nu .
        \]

        \item Let $\omega \in \bbOO_n$ and $\nu \in \optPolyFun^n 1_\omega$,
        i.e. $\nu \in \bbOO_{n+1}$ such that $\tgt \nu = \omega$. Then
        \[
            \mu_1 (\optPolyFun^n \eta_1) (\nu)
            = \mu_1 \left( \ytree{\nu} \biggraft_{[[p_i]]} \ytree{\ytree{\src_{[p_i]} \nu}} \right)
            = \tgt \left( \ytree{\nu} \biggraft_{[[p_i]]} \ytree{\ytree{\src_{[p_i]} \nu}} \right)
            = \nu ,
        \]
        where $[p_i]$ ranges over $\nu^\nodesymbol$.

        \item Akin to \cref{lemma:zn-zn}, one can show that elements of
        $\optPolyFun^n \optPolyFun^n \optPolyFun^n 1_\omega$ are
        $(n+2)$-opetopes $\xi$ of height $3$ such that $\tgt \tgt \xi =
        \omega$. Let $\xi$ be such an opetope, and write it as
        \[
            \xi = \ytree{\alpha} \biggraft_{[[p_i]]} \underbrace{\left(
                \ytree{\beta_i} \biggraft_{[[q_{i, j}]]} \ytree{\gamma_{i, j}}
            \right)}_{X_i}
            = \underbrace{\left(
                \ytree{\alpha} \biggraft_{[[p_i]]} \ytree{\beta_i}
            \right)}_{Y} \biggraft_{[[p_i][q_{i, j}]]} \ytree{\gamma_{i, j}}
        \]
        where $\alpha, \beta_i, \gamma_{i, j} \in \bbOO_n$, $[p_i]$ ranges over
        $\alpha^\nodesymbol$ and $[q_{i, j}]$ over $\beta_i^\nodesymbol$. Then
        \begin{align*}
            \mu_1 (\optPolyFun^n \mu_1) (\gamma)
            &= \tgt \left(
                \ytree{\alpha} \biggraft_{[[p_i]]} \ytree{\tgt X_i} \right) \\
            &= \tgt \left(
                \ytree{\tgt Y} \biggraft_{[\readdress_Y [[p_i][q_{i, j}]]]} \ytree{\gamma_{i, j}}
            \right) & \spadesuit \\
            &= \mu_1 \mu_{\optPolyFun^n 1} (\gamma) ,
        \end{align*}
        where $\spadesuit$ derives from the associativity axiom of the
        uncolored monad $\optPolyFun^n$.
    \end{enumerate}
\end{proof}

\begin{proof}
    [Proof of \cref{lemma:diagramatic-composite-diagramatic}]
    First, note that
    \begin{align*}
        \tgt (\xi_2 \subst_{[p_2]} \xi_1)
        &= \tgt \tgt (\ytree{\xi_2} \graft_{[[p_2]]} \ytree{\xi_1})
            & \text{by \cref{lemma:caf}} \\
        &= \tgt \src_{[]} (\ytree{\xi_2} \graft_{[[p_2]]} \ytree{\xi_1})
            & \condition{Glob2} \\
        &= \tgt \xi_2 = \omega_2 .
    \end{align*}
    Using \eqref{eq:node-decomposition}, we write
    \[
        \zeta
        = \xi_2 \subst_{[p_2]} \xi_1
        = \left( \alpha_2 \graft_{[p_2]} \alpha_1 \right) \graft_{[p_2 p_1]}
        \ytree{\omega_0} \biggraft_{[[q_i]]} \underbrace{\left(
            \beta_i \biggraft_{[l_{i, j}]} \gamma_{i, j}
        \right)}_{\delta_i}
    \]
    where $[q_i]$ ranges over $\omega_0^\nodesymbol$ and $[l_{i, j}]$ over
    $\beta_i^\leafsymbol$, such that
    \[
        \xi_1 =
        \alpha_1 \graft_{[p_1]} \ytree{\omega_0} \biggraft_{[[q_i]]} \beta_i,
        \qqquad
        \xi_2 =
        \alpha_2 \graft_{[p_2]} \ytree{\omega_1}
        \biggraft_{[\readdress_{\xi_1} [p_1[q_i]l_{i, j}]]} \gamma_{i, j}
    \]
    Since leaf addresses of $\xi_1$ are of the form $[p_1[q_i]l_{i, j}]$ for
    some $[q_i] \in \omega_0^\nodesymbol$ and $[l_{i, j}] \in
    \beta_i^\leafsymbol$, and since $\omega_1 = \tgt \xi_1$, the node addresses
    of $\omega_1$ are of the form $\readdress_{\xi_1} [p_1[q_i]l_{i, j}]$,
    which justifies the decomposition of $\xi_2$ above.

    Applying the definition of $\dotH$ we have, for $[q_i] \in
    \omega_0^\nodesymbol$, $[l_{i, j}] \in \beta_i^\leafsymbol$, and $[r] \in \gamma_{i, j}$,
    \begin{align*}
        \dotH \src_{[p_2 p_1]} :
            S [\omega_0]
            &\longrightarrow
            \optPolyFun S [\omega_2] \\
        (\dotH \src_{[p_2 p_1]}) [q_i] :
            S [\tgt \delta_i]
            &\longrightarrow
            S [\omega_2] \\
        \readdress_{\delta_i} [l_{i, j} r]
            &\longmapsto
            \readdress_{\zeta} [p_2 p_1 [q_i] l_{i, j} r] ;
            & \spadesuit \\
        \dotH \src_{[p_1]} :
            S [\omega_0]
            &\longrightarrow
            \optPolyFun S [\omega_1] \\
        (\dotH \src_{[p_1]}) [q_i] :
            S [\tgt \beta_i]
            &\longrightarrow
            S [\omega_1] \\
        \readdress_{\beta_i} [l_{i, j}]
            &\longmapsto
            \readdress_{\xi_1} [p_1 [q_i] l_{i, j}] ;
            & \diamondsuit \\
        \dotH \src_{[p_2]} :
            S [\omega_1]
            &\longrightarrow
            \optPolyFun S [\omega_2] \\
        (\dotH \src_{[p_2]}) (\readdress_{\xi_1} [p_1 [q_i] l_{i, j}]) :
            S [\tgt \gamma_{i, j}]
            &\longrightarrow
            S [\omega_2] \\
        \readdress_{\gamma_{i, j}} [r]
            &\longmapsto
            \readdress_{\xi_2} [p_2 \:\: \readdress_{\xi_1} [p_1 [q_i] l_{i, j}] \:\: r] .
            & \clubsuit
    \end{align*}
    Thus,
    \begin{align*}
        (\dotH \src_{[p_2 p_1]}) ([q_i])
            (\readdress_{\delta_i} [l_{i, j} r])
            &= \readdress_{\zeta} [p_2 p_1 [q_i] l_{i, j} r]
            & \text{by } \spadesuit \\
        &= \readdress_{\xi_2} [p_2 \:\: \readdress_{\xi_1} [p_1 [q_i] l_{i, j}] \:\: r]
            & \heartsuit \\
        &= (\dotH \src_{[p_2]})
            (\readdress_{\xi_1} [p_1 [q_i] l_{i, j}])
            (\readdress_{\gamma_{i, j}} [r])
            & \text{by } \clubsuit \\
        &= (\dotH \src_{[p_2]})
            \left(
                (\dotH \src_{[p_1]}) ([q_i]) (\readdress_{\beta_i} [l_{i, j}])
            \right)
            (\readdress_{\gamma_{i, j}} [r])
            & \text{by } \diamondsuit \\
        &= \left( \dotH \src_{[p_2]} \:\: \circ \:\: \dotH \src_{[p_1]}) \right)
            ([q_i]) (\readdress_{\delta_i} [l_{i, j} r]) ,
            & \scriptstyle{\maltese}
    \end{align*}
    where equality $\heartsuit$ comes from the monad structure on
    $\optPolyFun$, and $\scriptstyle{\maltese}$ from the definition of the
    composition in $\bbLambda$ when considered as the Kleisli category of
    $\optPolyFun$.
    \qedhere
\end{proof}

\begin{proof}
    [Proof of \cref{lemma:elementary-embeddings-diagramatic}]
    \begin{enumerate}
        \item Unfolding the definition of $\dotH \src_{[]}$ we get, for $[q]
        \omega^\nodesymbol$:
        \begin{align*}
            \dotH \src_{[]} :
                S [\ytree{\psi}]
                &\longrightarrow
                \optPolyFun S [\omega] \\
            (\dotH \src_{[]}) ([]) :
                S [\tgt \ytree{\omega}]
                &\longmapsto
                S [\omega] \\
            \readdress_{\ytree{\omega}} [[q]]
                &\longmapsto
                \readdress_\xi [[][q]] .
        \end{align*}
        Since $\readdress_{\ytree{\omega}} [[q]] = [q] = \readdress_\xi
        [[][q]]$, $\dotH \src_{[]}$ corresponds to the cell $\id_{S
        [\omega]} \in \optPolyFun S [\omega]_{\tgt \omega}$, thus is equal
        to $\dotH \tgt$ as claimed.

        \item First, graft trivial trees to $\xi$ so it's in the form of \cref{eq:node-decomposition}:
        \[
            \xi = \ytree{\barbeta} \graft_{[[p]]} \ytree{\beta}
                \biggraft_{[[q_i]]} \itree{\src_{[q_i]} \beta} ,
        \]
        where $[q_i]$ ranges over $\beta^\nodesymbol$. Unfolding the definition
        of $\dotH \src_{[p]}$ we get, for $[q] \beta^\nodesymbol$:
        \begin{align*}
            \dotH \src_{[p]} :
                S [\beta]
                &\longrightarrow
                \optPolyFun S [\omega] \\
            (\dotH \src_{[p]}) ([q]) :
                S [\tgt \itree{\src_{[q]} \beta}]
                &\longrightarrow
                S [\omega] \\
            [] = \readdress_{\itree{\src_{[q]} \beta}} []
                &\longmapsto
                \readdress_\xi [[p][q]] = [pq] ,
        \end{align*}
        thus $\dotH \src_{[p]} = \dotH i$.
        \qedhere
    \end{enumerate}
\end{proof}

\begin{proof}
    [Proof of \cref{lemma:diagramatic-lemma}]
    We proceed by induction on $\omega$.
    \begin{enumerate}
        \item Assume $\omega = \ytree{\psi}$ for some $\psi \in \bbOO_n$. Then
        \begin{align*}
        \bbLambda (\dotH \ytree{\psi}, \dotH \omega')
        &= \bbLambda (\optPolyFun S [ \ytree{\psi} ] , \optPolyFun S [ \omega' ] )  \\
        &\cong (\optPolyFun S [ \omega ])_\psi \ .
        \end{align*}
        Thus $f$ corresponds to a unique morphism $\tildF : S [ \nu ]
        \longrightarrow S [ \omega' ]$, for some $\nu \in \bbOO_{n+1}$
        such that $\tgt \nu = \psi$, and $f$ is the composite
        \[
            \dotH \ytree{\psi} = \dotH \psi \xrightarrow{\dotH \tgt} \dotH \nu
            \xrightarrow{\optPolyFun \tildF} \dotH \omega' .
        \]
        Those two arrows are diagrammatic by
        \cref{lemma:elementary-embeddings-diagramatic}, and thus $f$ is too by
        \cref{lemma:diagramatic-composite-diagramatic}.

        \item By induction, write $\omega = \nu_1 \graft_{[l]}
        \ytree{\psi_2}$ for some $\nu_1 \in \bbOO_{n+1}$, $[l] \in
        \nu_1^\leafsymbol$, and $\psi_2 \in \bbOO_n$. Write $\psi_1 = \tgt
        \nu_1$, and $\nu_2 = \ytree{\psi_2}$. Then $f$ restricts as $f_i$, $i =
        1, 2$, given by the composite $\dotH \nu_i \longrightarrow \dotH \omega
        \xrightarrow{f} \dotH \omega'$.

        Let $[l']$ be the edge address of $\omega'$ such that $\edg_{[l']}
        \omega' = f (\edg_{[l]} \omega)$. Then $\omega'$ decomposes as $\omega'
        = \beta_1 \graft_{[l']} \beta_2$, for some $\beta_1, \beta_2 \in
        \bbOO_{n+1}$ (in particular, $\beta_1$ and $\beta_2$ are sub
        $\optPolyFun^{n-1}$-trees of $\omega'$), and $f_1$ and $f_2$ factor as
        on the left
        \[
            \triangleURdiagram
                {\dotH \nu_i}
                {\dotH \beta_i}
                {\dotH \omega' ,}
                {\barF_i}
                {f_i}
                {u_i}
            \qqquad
            \frac{
                \begin{tikzcd} [ampersand replacement = \&]
                    \& \xi_i \\
                    \nu_i \ar[ur, sloped, near end, "\src_{[p_i]}"] \&
                    \beta_i \ar[u, "\tgt"]
                \end{tikzcd}
            }{
                \begin{tikzcd} [ampersand replacement = \&]
                    \dotH \nu_i \ar[r, "\barF_i"] \& \dotH \beta_i .
                \end{tikzcd}
            }
        \]
        where $u_i$ correspond to the subtree inclusion $\beta_i
        \longhookrightarrow \omega'$. By induction, $\barF_i$ is diagrammatic,
        say with the diagram on the right above, and thus $\beta_i$ can be
        written as $\beta_i = \barnu_i \subst_{[q_i]} \nu_i$, for some
        $\barnu_i \in \bbOO_{n+1}$ and $[q_i] \in \barnu_i^\nodesymbol$. In the
        case $i = 2$, note that $\beta_2 = \barnu_2 \subst_{[q_2]} \nu_2 =
        \barnu_2 \subst_{[q_2]} \ytree{\psi_2} = \barnu_2$.

        On the one hand we have
        \begin{align*}
            \edg_{[l']} \omega'
            &= f (\edg_{[l]} \omega)
                & \text{by def. of } [l'] \\
            &= f_1 (\edg_{[l]} \nu_1)
                & \text{since } \omega = \nu_1 \graft_{[l]} \ytree{\psi_2} \\
            &= u_1 \barF_1 (\edg_{[l]} \nu_1)
                & \text{since } f_1 = u_1 \barF_1 \\
            &= u_1 (\edg_{[q_1 l]} \beta_1)
                & \text{since } \beta_1 = \barnu_1 \subst_{[q_1]} \nu_1 \\
            &= \edg_{[q_1 l]} \omega ,
        \end{align*}
        showing $[l'] = [q_1 l]$, and thus that $\barnu_1$ is of the form
        \begin{equation}
            \label{eq:diagramatic-lemma:barnu1}
            \barnu_1 = \mu_1 \graft_{[q_1]} \ytree{\psi_1} \biggraft_{[[r_{1, j}]]} \mu_{1, j} ,
        \end{equation}
        where $[r_{1, j}]$ ranges over $\psi_1^\nodesymbol -
        \{\readdress_{\nu_1} [l] \}$, and $\mu_1, \mu_{1, j} \in \bbOO_{n+1}$.
        On the other hand,
        \begin{align*}
            \edg_{[l']} \omega'
            &= f (\edg_{[l]} \omega)
                & \text{by def. of } [l'] \\
            &= f_2 (\edg_{[]} \nu_2)
                & \text{since } \omega = \nu_1 \graft_{[l]} \ytree{\psi_2} \\
            &= u_2 \barF_2 (\edg_{[]} \nu_2)
                & \text{since } f_1 = u_2 \barF_2 \\
            &= u_2 (\edg_{[q_2]} \beta_2)
                & \text{since } \beta_2 = \barnu_2 \subst_{[q_2]} \nu_2 \\
            &= \edg_{[l']} \omega' ,
        \end{align*}
        showing $[q_2] = []$, and so $\src_{[]} \beta_2 = \src_{[]} \barnu_2 =
        \psi_2$, and we can write $\beta_2$ as
        \begin{equation}
            \label{eq:diagramatic-lemma:beta2}
            \beta_2 = \ytree{\psi_2} \biggraft_{[[r_{2, j}]]} \mu_{2, j} ,
        \end{equation}
        where $[r_{2, j}]$ ranges over $\psi_2^\nodesymbol$, and $\mu_{2, j}
        \in \bbOO_{n+1}$. Finally, we have
        \begin{align*}
            \omega'
            &= \beta_1 \graft_{[l']} \beta_2
            = (\barnu_1 \subst_{[q_1]} \nu_1) \graft_{[l']} \beta_2 \\
            &= \left(
                \mu_1 \graft_{[q_1]} \nu_1
                \biggraft_{\readdress^{-1}_{\nu_1} [r_{1, j}]} \mu_{1, j}
            \right) \graft_{[l']} \left(
                \ytree{\psi_2}
                \biggraft_{[[r_{2, j}]]} \mu_{2, j}
            \right)
                & \text{by \eqref{eq:diagramatic-lemma:barnu1} and \eqref{eq:diagramatic-lemma:beta2}} \\
            &= \left( \left(
                \mu_1 \graft_{[q_1]}
                \underbrace{\nu_1 \graft_{[l]} \ytree{\psi_2}}_{= \omega}
            \right)
                \biggraft_{[q_1] \cdot \readdress^{-1}_{\nu_1} [r_{1, j}]}
                    \mu_{1, j}
            \right)
                \biggraft_{[l'[r_{2, j}]]} \mu_{2, j}
                & \text{rearranging terms} \\
            &= \baromega \subst_{[q_1]} \omega ,
        \end{align*}
        for some $\baromega' \in \bbOO_{n+1}$.\footnote{Specifically,
        \[
            \omega'
            =
            \mu_1 \graft_{[q_1]} \left( \left(
                \ytree{\tgt \omega}
                \biggraft_{[\readdress_\omega \readdress^{-1}_{\nu_1} [r_{1, j}]]}
                    \mu_{1, j}
            \right)
                \biggraft_{[\readdress_\omega [l'[r_{2, j}]]]} \mu_{2, j}
            \right)
            =
            \left( \left(
                \mu_1 \graft_{[q_1]} \ytree{\tgt \omega}
            \right)
                \biggraft_{[q_1] \cdot [\readdress_\omega \readdress^{-1}_{\nu_1} [r_{1, j}]]}
                    \mu_{1, j}
            \right)
                \biggraft_{[q_1] \cdot [\readdress_\omega [l'[r_{2, j}]]]}
                    \mu_{2, j} .
        \]} Finally, by \cref{lemma:elementary-embeddings-diagramatic}, the following is a diagram of $\dotH f$:
        \[
            \frac{
            \begin{tikzcd} [ampersand replacement = \&]
                \& \xi \\
                \omega \ar[ur, sloped, near end, "\src_{[[q_1]]}"] \&
                \omega' \ar[u, "\tgt"]
            \end{tikzcd}
            }{
            \begin{tikzcd} [ampersand replacement = \&]
                \dotH \omega \ar[r, "f"] \& \dotH \omega' ,
            \end{tikzcd}
            }
            \qqquad
            \xi \eqdef \ytree{\baromega} \graft_{[[q_1]]} \ytree{\omega} .
        \]
        \qedhere
    \end{enumerate}
\end{proof}


\printindex

\pagebreak

\bibliographystyle{alpha}
\bibliography{bibliography}

\end{document}